\documentclass[paper=a4, fontsize=11pt, abstract=on]{scrartcl}

\usepackage{amsmath}
\usepackage{amsthm}
\usepackage{amsfonts}
\usepackage{amssymb}
\usepackage{mathtools}
\usepackage{hyperref}
\usepackage{titling}
\usepackage{framed}
\usepackage{dsfont}
\usepackage{mathrsfs}
\usepackage{color}

\usepackage[headsepline]{scrlayer-scrpage}
\usepackage{geometry}

\usepackage[utf8]{inputenc}

\usepackage{lmodern}
\usepackage[T1]{fontenc}

\usepackage{etoolbox}
\patchcmd{\thebibliography}
{\settowidth}
{\setlength{\itemsep}{0pt plus 0.1pt}\settowidth}
{}{}
\apptocmd{\thebibliography}
{\small}
{}{}

\title{\textbf{ Optimal Gamma Approximation on Wiener Space}}
\author{
	E. Azmoodeh\thanks{Ruhr University Bochum, Faculty of Mathematics, IB 2/101, 44780 Bochum, Germany. E-mail: ehsan.azmoodeh@rub.de},
	P. Eichelsbacher\thanks{Ruhr University Bochum, Faculty of Mathematics, IB 2/115, 44780 Bochum, Germany. E-mail: peter.eichelsbacher@rub.de}
	and L. Knichel\thanks{Ruhr University Bochum, Faculty of Mathematics, IB 2/95, 44780 Bochum, Germany. E-mail: lukas.knichel@rub.de.
		Lukas Knichel has been supported by the German Research Foundation (DFG) via Research Training Group RTG 2131 \textit{High dimensional phenomena in probability -- fluctuations and discontinuity}}
}
\date{\today}
\theoremstyle{plain}
\newtheorem{Thm}{Theorem}[section]
\newtheorem{thm}[Thm]{Theorem}
\newtheorem{Lem}[Thm]{Lemma}
\newtheorem{lem}[Thm]{Lemma}
\newtheorem{Prop}[Thm]{Proposition}
\newtheorem{prop}[Thm]{Proposition}

\theoremstyle{definition}

\newtheorem{rem}[Thm]{Remark}

\newtheorem{ex}[Thm]{Example}

\def\E{\mathbb{E}}
\def\R{\mathbb{R}}

\def\N{\mathbb{N}}
\newcommand{\HH}{\mathfrak{H}}

\newcommand{\cont}[1]{\mathbin{\otimes_{#1}}}
\newcommand{\contIterated}[2]{\mathbin{\otimes_{#1}^{(#2)}}}
\newcommand{\scont}[1]{\mathbin{\widetilde{\otimes}_{#1}}}
\newcommand{\tensor}{\mathbin{\otimes}}
\newcommand{\symtensor}{\mathbin{\widetilde{\otimes}}}

\DeclareMathOperator{\Var}{Var}
\DeclareMathOperator{\Ker}{Ker}
\DeclareMathOperator{\Tr}{Tr}

\DeclareMathOperator{\CenteredGamma}{\overline{\Gamma}}
\DeclareMathOperator{\confhyper}{{}_{1}F _{1}}
\DeclarePairedDelimiter\sprod{\langle}{\rangle}
\DeclarePairedDelimiter\abs{\lvert}{\rvert}
\DeclarePairedDelimiter\norm{\lVert}{\rVert}
\DeclarePairedDelimiter\ceil{\lceil}{\rceil}

\newcommand{\ind}[1]{\mathds{1}_{\{ #1 \}}}

\let\temp\epsilon \let\epsilon\varepsilon \let\varepsilon\temp
\def\geq{\geqslant}
\def\leq{\leqslant}

\date{  }
\begin{document}
	\maketitle
	\begin{abstract}
		In \cite{n-p-noncentral}, Nourdin and Peccati established a neat characterization of Gamma approximation on a fixed Wiener chaos in terms of convergence of only the third and fourth cumulants.  In this paper, we provide an optimal rate of convergence in the $d_2$-distance in terms of the maximum of the third and fourth cumulants analogous to the result for normal approximation in \cite{n-p-optimal}. In order to achieve our goal, we introduce a novel operator theory approach to Stein's method. The recent development in Stein's method for the Gamma distribution of D\"obler and Peccati (\cite{d-p}) plays a pivotal role in our analysis. Several examples in the context of quadratic forms are considered to illustrate our optimal bound.  
	\end{abstract}

	\vskip0.3cm
	\noindent \textbf{Keywords}:
	Gamma approximation, Wiener chaos, Cumulants/Moments, Weak convergence, Malliavin Calculus, Berry--Esseen bounds, Stein's method, Wasserstein distances, Quadratic form\\
	\noindent 
	\textbf{MSC 2010}: 60F05, 60G50, 60H07
	
		\tableofcontents

\section{Introduction and Main Result}
Let $X=\{X(h) : h \in \HH \}$ be an isonormal Gaussian process over a separable Hilbert space $\HH$ on a suitable probability space $(\Omega, \mathscr{F}, P)$. In the landmark article \cite{FmtOriginalReference} Nualart and Peccati discovered an astonishing central limit theorem (CLT) known nowadays as the \textit{fourth moment theorem} for a sequence of normalized random variables  inside a fixed Wiener chaos associated to $X$. It states that the convergence in distribution towards a standard Gaussian distribution is equivalent to the sole requirement that the fourth moments converge to $3$. A few years later, their findings have created a fertile line of research, culminating in the popular article \cite{StMethOnWienChaos}, introducing the so called \textit{Malliavin-Stein} approach, 
an elegant combination of two probabilistic techniques namely Stein method \cite{stein,ChenGoldsteinShao2011} and Malliavin calculus \cite{GelbesBuch,Nua-Nua} in order to quantify the probability distance between a square integrable
Wiener functional and a normal distribution. The reader may consult the excellent monograph \cite{n-pe-1}, as well as the constantly updated web resource \texttt{\url{ https://sites.google.com/site/malliavinstein/home}} for a huge amount of applications and generalizations of the aforementioned results. Our study is mainly inspired by the following discovery (item (b) of the forthcoming theorem), which presents an optimal version of the fourth moment theorem. For every real-valued random variable $F$ the quantity $\kappa_r (F)$ stands for the $r$th cumulant of $F$, see section \ref{sec:cumulant}.
\begin{thm}[\textbf{(Optimal) fourth moment theorem} \cite{FmtOriginalReference,StMethOnWienChaos,n-p-optimal}] \label{thm:optimal-normal}
Fix $q \ge 2$. Let $\{ F_n : n\ge 1\}$ be a sequence of random variables in the $q$th Wiener chaos associated to $X$ such that $\E[F^2_n]=1$ for every $n \in \N$. Then
\begin{description}
	\item[(a)] $F_n \to N \sim \mathscr{N}(0,1)$ in distribution if and only if $\E[F^4_n] \to 3$. Also, the following quantitative estimate is in order: for $n\ge 1$,
	\begin{equation}\label{eq:4MT-quantitative}
	d_{TV} (F_n, N) \le 2 \sqrt{\frac{q-1}{3q}} \, \sqrt{\vert \kappa_4 (F_n) \vert}.
	\end{equation} 
	\item[(b)] Under the assumptions of item (a) there exist two constants $C_1$ and $C_2$ (independent of $n$) such that the following optimal rate of convergence in total variation distance holds:
\[ C_1 \,  \max\{ \abs{\kappa_3(F_n)}, \abs{\kappa_4(F_n)}  \} \leq d_{TV}(F_n,N) \leq  C_2 \, \max\{ \abs{\kappa_3(F_n)}, \abs{\kappa_4(F_n)}  \}.\]
\end{description}
\end{thm}

Fix a parameter $\nu>0$. In this paper, the target distribution of interest is the so called \textit{centered Gamma} distribution denoted by $G(\nu) \sim CenteredGamma(\nu)$. This means that $G(\nu) = 2 \, \widehat{G}(\nu /2) - \nu$, where $\widehat{G}(\nu /2)$ is a standard Gamma random variable with density $\widehat{g}(x) = x^{\frac{\nu}{2} -1} \, e^{-x} \, \Gamma(\frac{\nu}{2})^{-1} \, \mathds{1}_{(0, \infty)}(x)$. Here $\Gamma(\nu):= \int_{0}^{+\infty} x^{\nu-1} e^{-x} dx$ denotes the \textit{Euler Gamma function}. The centered Gamma distribution frequently appears as a natural limiting distribution in the context of the fourth moment theorems in several studies, see for example \cite{a-c-p,a-m-m-p,a-p-p,Ciprian-1,Ciprian-2,A-S-Gamma,VarianceGammaPaper,ledoux4MT,n-r,n-poly-2wiener,a-m-p-s, EdenViquez2015}. Our principal goal is to provide an optimal rate (analogous to that of item (b) Theorem \ref{thm:optimal-normal}) for the Gamma approximation on a fixed Wiener chaos. The statement of the next result is an up-to-date significant improvement over the years of the findings in \cite{n-p-noncentral,StMethOnWienChaos,InvPrinForHomSums,d-p}. 
 \begin{thm}\label{thm:Gamma-approximation-Int}
 Let $\nu >0$. Fix $q \ge 2$ an even number (see \cite[Remark~1.3,~item 3]{n-p-noncentral} when $q$ is odd). Assume $F=I_q(f)$ is a random element in the $q$th Wiener chaos such that $\E[F^2]=2\nu$. Then there exists a constant $C_{\nu,q}$ (may depend on $\nu$ and $q$) such that 
 \begin{equation}\label{eq:Gamma-Rate-Int}
 \begin{split}
 d_1 (F,G(\nu)) &\le C_{\nu,q} \, \sqrt{  \Big \vert \left(\kappa_4 (F) - \kappa_4 (G(\nu)) \right) - 12 \left(   \kappa_3(F) - \kappa_3 (G(\nu) \right) \Big \vert  }
  \le C'_{\nu,q} \sqrt{ \mathbf{M}(F)}
 \end{split}
 \end{equation}	
 where
 \begin{equation}\label{eq:M}
 \mathbf{M}(F):= \max \Big \{  \Big \vert  \kappa_4 (F) - \kappa_4 (G(\nu))  \Big \vert ,  \Big \vert  \kappa_3 (F) - \kappa_3 (G(\nu))  \Big \vert \Big \}.
 \end{equation}
 Here $d_1$ stands for the so called $1$-Wasserstein metric (see below for definition). As a consequence, for a sequence $\{F_n: n\ge 1\}$ of random variables in the $q$th Wiener chaos such that $\E[F^2_n]=2\nu$ for every $n \in \N$, the following remarkable equivalence of asymptotic statements are in order:
 \begin{description}
 	\item[(a)] $F_n \rightarrow G(\nu)$ in distribution. 
 	\item[(b)] $\kappa_3(F_n) \to 8\nu$, and $\kappa_4 (F_n) \to 48 \nu$.
 \end{description}
 \end{thm}
 The exact shape of the constant $C_{\nu,q}$ can be found in the aforementioned references. Note that $\kappa_3(G(\nu)), \kappa_4(G(\nu)) \neq 0$ unlike the case of normal approximation. We also recall the following natural generalization of the $1$-Wasserstein metric $d_1$ that we will make use of throughout the paper. Let $X$ and $Y$ be two real-valued random variables. For $k \ge 2$, define  
 \begin{equation*}
 d_k(X,Y) := \sup_{h \in \mathcal{H}_k} \abs[\Big]{ \E[h(X)] - \E[h(Y)]  } \label{eq:d_kMetricsDefinition}
 \end{equation*}
 where the class of the test functions is $ \mathcal{H}_k := \{ h \in C^{k-1}(\R) : h^{(k-1)} \in \operatorname{Lip}(\R) \text{ and } \norm{h^{(1)}}_\infty \leq 1,  \ldots, \norm{h^{(k)}}_\infty \leq 1 \}$. Here, $\norm{h^{(k)}}_{\infty}$ denotes the smallest Lipschitz constant of $h^{(k-1)}$, see \eqref{eq:SmallestLipschitzConstant}. A significant and also very challenging question, which we will deal with in this paper, is whether one can either provide an optimal rate or improve the rate \eqref{eq:Gamma-Rate-Int} available in Theorem \ref{thm:Gamma-approximation-Int}. For a general sequence $\{F_n: n\ge 1\}$ and a suitable probability metric $d$ (often we assume that the topology induced by metric $d$ is stronger than convergence in distribution), following \cite[Definition 9.2.1]{n-pe-1}, we say that a numerical sequence $\{\rho(n): n \in \N \}$ of strictly positive real numbers,  decreasing to $0$, yields an \textit{optimal rate} with respect to the metric $d$, if there exist two constants $C_1$ and $C_2$ (independent of $n$) such that 
 \begin{equation*}
 C_1 \le \frac{d(F_n,G(\nu))}{\rho(n)} \le C_2, \, \forall \, n \in \N.
 \end{equation*}
 Our main result is the following non asymptotic optimal Gamma approximation within the second Wiener chaos that improves upon the rate \eqref{eq:Gamma-Rate-Int} by a square power. 
 
\begin{thm}[\textbf{Non asymptotic optimal Gamma approximation}]\label{thm:Main-Result-Int}
Let $\nu>0$, and $G(\nu) \sim CenteredGamma(\nu)$. Assume that $F$ is a random variable in the second Wiener chaos associated with $X$, such that $\E[F^2]=2\nu$. Then there exist two constants $0<C_1<C_2$ (possibly depending on the parameter $\nu$) such that 
\begin{equation}\label{eq:optimal-rate}
C_1 \, \mathbf{M} (F) \le d_{2} (F, G(\nu)) \le C_2 \,  \mathbf{M} (F),
\end{equation}	
where the quantity $\mathbf{M}(F)$ is given by \eqref{eq:M}.
\end{thm}
 
\begin{rem} 
	\begin{description}
	\setlength\itemsep{-0.3em}
	\item[(a)] A significant feature of the optimal rate \eqref{eq:optimal-rate}, unlike the one in item (b) of Theorem \ref{thm:optimal-normal} in the normal approximation case, is that it is non asymptotic and a priori does not assume the law of the chaotic random variable $F$ to be close to that of $G(\nu)$. 
	\item[(b)] For the upper bound, the starting point is an adaption of the technique developed in \cite{n-p-optimal}. However, in order to achieve the optimal upper bound we introduce a novel technique within Stein's method to split test functions relying on tools from operator theory. This is the topic of section \ref{sec:operator-theory}.
	\item[(c)] Our methodology to obtain the optimal lower bound is based on complex analysis and differs from that in \cite{n-p-optimal}. Up to our knowledge this method is new. 
	\item[(d)] Theorem \ref{thm:Main-Result-Int} has to be seen as a full generalization of the main findings of \cite{OurGammaPaper}, where we assumed some additional technical conditions.
	\end{description}
	\end{rem}

The outline of our paper is as follows: In section $2$, we give a brief introduction to Mallia\-vin calculus on the Wiener space and specify the notation used in the paper. Section $3$ gathers the essential ingredients of Stein's method for the centered Gamma distribution, developed recently in \cite{d-p}. 
Section $4$ contains the main theoretical findings of this paper -- an upper bound for the $d_2$ distance between a general element $F$ living in a finite sum of Wiener chaoses and the target distribution $G(\nu)$ in terms of iterated Gamma operators, as well as the optimal Gamma approximation rate. The end of this section is devoted to applications of our main findings.  Lastly, we close the paper with an appendix section with focus on the newly introduced Gamma operators.  

\section{Preliminaries: Gaussian Analysis and Malliavin Calculus}
In this section, we provide a brief introduction to Malliavin calculus and define some of the operators used in this framework. For more details, see for example the textbooks \cite{n-pe-1,GelbesBuch,Nua-Nua}.

\subsection{Isonormal Gaussian Processes and Wiener Chaos}
Let $\mathfrak{H}$ be a real separable Hilbert space with inner product $\sprod{\cdot,\cdot}_{\mathfrak{H}}$, and $X = \{X(h) : h \in \mathfrak{H} \}$ be an isonormal Gaussian process, defined on some probability space $(\Omega, \mathscr{F},P)$. This means that $X$ is a family of centered, jointly Gaussian random variables with covariance structure $\E[X(g)X(h)] = \sprod{g,h}_{\mathfrak{H}}$. We assume that $\mathscr{F}$ is the $\sigma$-algebra generated by $X$. For an integer $q \geq 1$, we will write $\mathfrak{H}^{\otimes q}$ or $\mathfrak{H}^{\odot q}$ to denote the $q$-th tensor product of $\mathfrak{H}$, or its symmetric $q$-th tensor product, respectively. If $H_q(x) =(-1)^{q}e^{x^{2}/2}{\frac {d^q}{d x^n}}e^{-x^{2}/2}$ is the $q$-th Hermite polynomial, then the closed linear subspace of $L^2(\Omega)$ generated by the family $\{H_q(X(h)) : h \in \mathfrak{H}, \norm{h}_\mathfrak{H} = 1 \}$ is called the $q$-th \textit{Wiener chaos} of $X$ and will be denoted by $\mathscr{H}_q$. For $f \in \mathfrak{H}^{\odot q}$, let $I_q(f)$ be the $q$-th multiple Wiener-Itô integral of $f$ (see \cite[Definition~2.7.1]{n-pe-1}). An important observation is that for any $f \in \mathfrak{H}$ with $\norm{f}_{\mathfrak{H}}=1$ we have that $H_q(X(f)) = I_q(f^{\otimes q})$. As a consequence $I_q$ provides an isometry from $\mathfrak{H}^{\odot q}$ onto the $q$-th Wiener chaos $\mathscr{H}_q$ of $X$. It is a well-known fact, called the \textit{Wiener-Itô chaotic decomposition}, that any element $F \in L^2(\Omega)$ admits the expansion
\begin{equation}
F = \sum_{q=0}^{\infty} I_q(f_q), \label{eq:ChaoticExpansion}
\end{equation}
where $f_0 = \E[F]$ and the $f_q \in \mathfrak{H}^{\odot q}$, $q \geq 1$ are uniquely determined. An important result is the following isometry property of multiple integrals. Let $f \in \mathfrak{H}^{\odot p}$ and $g \in \mathfrak{H}^{\odot q}$, where $1 \leq q \leq p$. Then
\begin{equation}
\E[ I_p(f) I_q(g) ] = \begin{cases}
p! \, \sprod{f,g}_{\mathfrak{H}^{\otimes p}}  & \text{if } p= q\\
0 & \text{otherwise}.
\end{cases} \label{eq:IsometryProperty}
\end{equation}

\subsection{The Malliavin Operators}

We denote by $\mathscr{S}$ the set of \textit{smooth} random variables, i.e. all random variables of the form $F= g(X(\varphi_1),\ldots, X(\varphi_n))$, where $n \geq 1$, $\varphi_1, \ldots, \varphi_n \in \HH$ and $g:\R^n \to \R$ is a $C^{\infty}$-function, whose partial derivatives have at most polynomial growth. For these random variables, we define the \textit{Malliavin derivative} of $F$ with respect to $X$ as the $\HH$-valued random element $DF \in L^2(\Omega,\HH)$ defined as
\[ DF = \sum_{i=1}^{\infty} \frac{\partial g}{\partial x_i} \big( X(\varphi_1), \ldots, X(\varphi_n) \big) \, \varphi_i. \]
The set $\mathscr{S}$ is dense in $L^2(\Omega)$ and using a closure argument, we can extend the domain of $D$ to $\mathbb{D}^{1,2}$, which is the closure of $\mathscr{S}$ in $L^2(\Omega)$ with respect to the norm $\norm{F}_{\mathbb{D}^{1,2}} := \E[F^2] + \E[ \norm{DF}_{\HH}^2 ]$. See \cite{n-pe-1} for a more general definition of higher order Malliavin derivatives and spaces $\mathbb{D}^{p,q}$. The Malliavin derivative satisfies the following chain-rule. If $\phi:\R^m \to \R$ is a continuously differentiable function with bounded partial derivatives and $F=(F_1, \ldots, F_m)$ is a vector of elements of $\mathbb{D}^{1,q}$ for some $q$, then $\phi(F)\in \mathbb{D}^{1,q}$ and
\begin{equation}
D \phi(F) = \sum_{i=1}^{m} \frac{\partial \phi}{\partial x_i} (F) \, D F_i. \label{eq:ChainRule}
\end{equation}
Note that the conditions on $\phi$ are not optimal and can be weakened. For $F \in L^2(\Omega)$, with chaotic expansion as in \eqref{eq:ChaoticExpansion}, we define the \textit{pseudo-inverse} of the infinitesimal generator of the Ornstein-Uhlenbeck semigroup as
\[ L^{-1} F = - \sum_{p=1}^{\infty} \frac{1}{p} I_p(f_p). \]

The following integration by parts formula is one of the main ingredients to proving the main theorem of section \ref{sec:MainUpperBound}. Let $F,G \in \mathbb{D}^{1,2}$. Then 
\begin{equation}
\E[FG] = \E[F] \E[G] + \E[ \sprod{DG, -DL^{-1}F}_{\HH} ].  \label{eq:IntegrationByParts}
\end{equation}

\subsection{Gamma Operators and Cumulants}\label{sec:cumulant}

Let $F$ be a random variable with characteristic function $\phi_F(t) = \E[ e^{itF}]$. We define its $n$-th cumulant, denoted by  $\kappa_n(F)$, as
\[ \kappa_n(F) = \frac{1}{i^n} \frac{\partial^n}{\partial t^n} \log \phi_F(t) \Big\vert_{t=0}. \]
Let $F$ be a random variable with a finite chaos expansion. We define the operators $\Gamma_i$, $i \in \N_0$ via $\Gamma_0(F) := F$ and
\begin{equation}
\Gamma_{i+1} (F) := \sprod{D \Gamma_{i}(F) , -D L^{-1} F}_{\mathfrak{H}}, \quad \text{for } i\geq 0. \label{eq:GammOperatorDefinition}
\end{equation}
This is the Gamma operator used in the proof of the main theorem in \cite{n-p-optimal}, although it is defined differently there.
Note that there is also an alternative definition, which can be found in most other papers in this framework, see for example Definition 8.4.1 in \cite{n-pe-1} or Definition 3.6 in \cite{OptBerryEsseenRates}. For the sake of completeness, we also mention the classical Gamma operators, which we also call \textit{alternative} Gamma operators, which we shall denote by $\Gamma_{alt}$. These are defined via
\begin{equation}
\Gamma_{alt, 0}(F) := F \quad \text{and} \quad \Gamma_{alt, i+1} (F) := \sprod{D F , -D L^{-1} \Gamma_{alt, i}(F)}_{\mathfrak{H}}, \quad \text{for } i\geq 0. \label{eq:GammaOperatorAltDefinition}
\end{equation}
The classical Gamma operators are related to the cumulants of $F$ by the following identity from \cite{CumOnTheWienerSpace}: For all $j \geq 0$, we have
\[ \E[\Gamma_{alt, j}(F)] = \frac{1}{j!} \kappa_{j+1}(F). \]
If $j \geq 3$, this does not hold anymore for our new Gamma operators. Instead, in our next result, we will list some useful relations between the classical and the new Gamma operators.

\begin{Prop} \label{Prop:RelationOldAndNewGamma}
	Let $F$ be a centered random variable admitting a finite chaos expansion. Then
	\begin{itemize}
		\item[(a)] $\Gamma_1(F) = \Gamma_{alt,1}(F)$,
		\item[(b)] 
		$ \E\big[ \Gamma_j(F) \big] = \E\big[ \Gamma_{alt,j}(F) \big] = \frac{1}{j!} \kappa_{j+1}(F)$ for $j=1,2$.
		\item[(c)] $ \E\big[ \Gamma_3(F) \big] = 2 \, \E \big[ \Gamma_{alt,3}(F) \big] - \Var\big(\Gamma_1(F) \big) = \frac{1}{3} \kappa_4(F) - \Var\big(\Gamma_1(F) \big)$,
		\item[(d)] When $F= I_2(f)$, for some $f \in \HH^{\odot 2}$, is an element of the second Wiener chaos, then
		\[ \Gamma_j(F) = \Gamma_{alt,j}(F) \quad \text{for all } j \geq 1. \]
	\end{itemize}
\end{Prop}
The proofs of these statements can be found in the appendix along with an explicit representation of the Gamma operators in terms of contractions.

\subsection{Useful facts on Second Wiener Chaos}\label{sec:2wiener}

Let $F=I_2(f)$, for some $f \in \mathfrak{H}^{\odot 2}$ be a generic element in the second Wiener chaos. It is a classical result (see \cite[section 2.7.4]{n-pe-1}) that these kind of random variables can be analyzed through the associated \textit{Hilbert-Schmidt operator} $A_f : \mathfrak{H} \to \mathfrak{H}$ that maps $g \mapsto f \cont{1} g$. Denote by $\{ c_{f,i} : i \in \N \}$ the set of eigenvalues of $A_f$. We also introduce the following sequence of auxiliary kernels $\Big\{ f \contIterated{1}{p} f : p \geq 1 \Big\} \subset \HH^{\odot 2}$, defined recursively as $ f \contIterated{1}{1} f = f$, and, for $p \geq 2$,
$ f \contIterated{1}{p} f = \Big( f \contIterated{1}{p-1} f \Big) \cont{1} f$.  

\begin{Prop} (see e.g. \cite[p.~43]{n-pe-1})\mbox{} \\[-1em]\label{Prop:2choas-properties}
	\begin{enumerate}
		\item The random element $F$ admits the representation \begin{equation} F = \sum_{i=1}^{\infty} c_{f,i} \left( N_i^2 - 1 \right), \label{eq:SecondWienerChaosEigenvalueRepresentation} \end{equation}
		where the $(N_i)$ are i.i.d. $\mathscr{N}(0,1)$ and the series converges in $L^2(\Omega)$ and almost surely.
		\item For every $p \geq 2$
		\begin{equation}
		\begin{split} 
		\kappa_p(F)  &=2^{p-1} (p-1)! \sum_{i=1}^{\infty} c_{f,i}^p  = 2^{p-1} (p-1)!  \langle f, f \contIterated{1}{p-1} f \rangle_{\HH}\\
		&= 2^{p-1} (p-1)! \Tr \left( A^p_f \right) \label{eq:SecondWienerChaosCumulantFormulaEigenvalues} \end{split}\end{equation}
		where $ \Tr ( A^p_f )$ stands for the trace of the $p$th power of operator $A_f$.
	\end{enumerate}
\end{Prop}

It is known that when $\nu$ is an integer, $G(\nu)\sim \overline{\chi}^2$ is a centered chi-squared random variable with $\nu$ degrees of freedom, and \eqref{eq:SecondWienerChaosEigenvalueRepresentation} shows that $G(\nu)$ is itself an element of the second Wiener chaos, where $\nu$-many of the eigenvalues are $1$ and the remaining ones are $0$. Hence, in this case, we deduce from \eqref{eq:SecondWienerChaosCumulantFormulaEigenvalues} that $\kappa_p(G(\nu)) = 2^{p-1} (p-1)! \, \nu$. Perhaps not surprisingly, this is also the case when $\nu$ is any positive real number.

\begin{Lem}
	Let $\nu>0$ and $G(\nu) \sim CenteredGamma(\nu)$. Then 
	\begin{equation} 
	\kappa_p(G(\nu)) = \begin{dcases}
	0 & , p=1; \\
	2^{p-1} (p-1)! \, \nu &, p \geq 2.
	\end{dcases}  \label{eq:CenteredGammaCumulants}
	\end{equation}
\end{Lem}
\begin{proof}
	Since the cumulant generating function of a Gamma random variable is well-known, we can easily compute that of $G(\nu)$ to be $K(t) = \frac{\nu}{2} \log \left( \frac{1}{1- 2t} \right) - \nu t$.	By simple induction over $p$, we obtain
	\[ \frac{\mathrm{d}^p K}{\mathrm{d} t^p} (t) = \begin{dcases}
	- \nu + \frac{\nu}{1-2t} & , p=1; \\
	\frac{\nu}{2} \frac{2^p (p-1)!}{(1-2t)^{p+1}}& , p \geq 2.
	\end{dcases}   \]
	The result now follows by letting $t=0$.
\end{proof}

\begin{Lem} \label{lem:Var(Gamma_r-2Gamma_r-1)}
	Let $F=I_2(f)$ for some $f \in \HH^{\odot 2}$, and denote by $A_f$ the corresponding Hilbert-Schmidt operator with eigenvalues $\{ c_{f,i} : i \geq 1 \}$. Then for every $r \geq 1$,
	\begin{align*}
	\Var \Big( \Gamma_r(F) -  2 \Gamma_{r-1}(F) \Big) & = 2^{2r+1} \sum_{i=1}^{\infty} c_{f,i}^{2r} ( c_{f,i} - 1 )^2 \\
	& = \frac{1}{(2r+1)!}  \kappa_{2r+2}(F) - \frac{4}{(2r)!}  \kappa_{2r+1}(F) + \frac{4}{(2r-1)!}  \kappa_{2r}(F).
	\end{align*}
\end{Lem}

\begin{proof}
	From \cite{a-p-p} equation (24), which follows by induction on $r$, we have the representation 
	\begin{equation}\label{eq:CentredGammaInTermsOfContraction}
	\CenteredGamma_r(F) = 2^r I_2 \Big(f \contIterated{1}{r+1} f \Big).
	\end{equation}
	Using the isometry property \eqref{eq:IsometryProperty}, we obtain
	\begin{align*}
	\Var \Big( \Gamma_r(F) - & 2 \Gamma_{r-1}(F) \Big) = 2^{2r+1} \, \norm{ f \contIterated{1}{r+1} f - f \contIterated{1}{r} f }_{\HH^{\otimes 2}}^{2}  \\
	& = 2^{2r+1} \Big( \sprod{f, f \contIterated{1}{2r+1} f}_{\HH^{\otimes 2}} - 2 \, \sprod{f, f \contIterated{1}{2r} f}_{\HH^{\otimes 2}} + \sprod{f, f \contIterated{1}{2r-1} f}_{\HH^{\otimes 2}} \Big)  \\
	& = 2^{2r+1} \Tr \Big( A_f^{2r+2} - 2 \, A_f^{2r+1} + A_f^{2r} \Big).
	\end{align*}
	The result now follows with \eqref{eq:SecondWienerChaosCumulantFormulaEigenvalues}.
\end{proof}

\section{Stein's Method for the centered Gamma distribution }\label{sec:operator-theory}
Let $X_r \sim \Gamma(r,1)$ be distributed according to a Gamma distribution with \textit{shape parameter} $r>0$.  It means that random variable $X_r$ admits the density 
\begin{equation}
p_r(x) = \begin{cases} \label{eq:GammaPDF}
\frac{1}{\Gamma(r)} x^{r-1} e^{-x}, & \text{if } x >0,\\
0, & \text{otherwise}.
\end{cases}
\end{equation}
Consider the centered Gamma random variable $G(\nu) = 2 \, X_{\nu/2} - \nu  \sim CenteredGamma(\nu) $. Stein's method for $X_{\nu/2}$ has first been studied in \cite{LukThesis} and then later been refined in \cite{PickettThesis}.
It is well known (see e.g. \cite[equation~(24)]{d-p}) that the Stein equation for the centered Gamma random variable $G(\nu)$ associated to the test function $h$ is given by the following first order ODE with polynomial coefficients 

\begin{equation}\label{eq:CenteredGamma-stein-eq}
2(x+\nu) f'(x) - x f(x) = h(x) - \E \left[   h(G(\nu)) \right],
\end{equation}
where $h:\R \to \R$ is measurable and $\E \vert h(G(\nu)) \vert < \infty$.  The following result is taken from \cite[Theorem 2.3]{d-p} and plays a crucial role in our analysis. For the reader's convenience we restate it here. We also need the following convention that for every function $f:\R \to \R$ the quantity $ \norm{f'}_\infty$ stands for the smallest Lipschitz constant, i.e.
\begin{equation} \label{eq:SmallestLipschitzConstant}
\norm{f'}_{\infty} = \sup_{\substack{ x,y \in \R \\ x \neq y}} \frac{\abs{f(x) - f(y)} }{\abs{x-y}} \in \R \cup \{ +\infty\}.
\end{equation}
It is worth pointing out that $ \norm{f'}_\infty$ coincides with the uniform norm of the derivative of $f$ whenever $f$ is differentiable. 
\begin{thm}\label{thm:Gamma-Stein-Solution-Properties}(\cite[Theorem 2.3]{d-p})
(a) Let $h$ be a Lipschitz-continuous function on the  whole real line $\R$. Then there exists a unique bounded Lipschitz-continuous solution $S(h)$ to the equation \eqref{eq:CenteredGamma-stein-eq} on the whole real line $\R$ satisfying the bounds
$$ \big \Vert S(h) \big \Vert_\infty \le \Vert h' \Vert_\infty, \quad \text{and} \quad \big \Vert S(h)' \big \Vert_\infty \le c_\nu \Vert h' \Vert_\infty,$$
where the constant $c_\nu = \max \{ 1,  \frac{2}{\nu} \}$.  \\
(b) Suppose that the function $h$ is continuously differentiable on $\R$ such that both $h$ and $h'$ are Lipschitz-continuous. Then there is a continuously differentiable solution $S(h)$ of equation \eqref{eq:CenteredGamma-stein-eq} on $\R$ whose derivative $S(h)'$ is Lipschitz-continuous, and moreover 
$$ \big \Vert S(h)'' \big \Vert_\infty \le 	c_\nu \Vert h' \Vert_\infty + \Vert h''\Vert_\infty.$$
	\end{thm}

\subsection{Explicit Formula for the Solution of the Stein Equation}
This section is entirely based on \cite{d-p}. 
It is known that a Stein equation for the $\Gamma(r,1)$ distribution is given by
\begin{equation} \label{eq:GammSteinEq}
x f'(x) + (r-x) f(x) = h(x) - \E[h(X_{r})],
\end{equation}
where $h:\R \to \R$ is a measurable test function with $\E \vert h(X_r) \vert < +\infty$. D\"obler and Peccati \cite[p.~3406]{d-p} showed that if $h \in \operatorname{Lip}(\R)$, then there exists a unique Lipschitz-continuous function $f_h$ on $\R$ solving \eqref{eq:GammSteinEq}, given by 
\[
f_h(x) = \begin{cases}
f_h^-(x), & x < 0, \\
f_j^+(x), & x > 0,
\end{cases}
\]
where for $x < 0$, $f_h^-(x) = 
\frac{1}{x q_l(x)} \int_{0}^{x} \Big( h(t) - \E \big[ h(X_r) \big] \Big) q_l(t)  dt$ and $q_l(x) = -(-x)^{r-1} e^{-x}$. Also 
$ f_h^+ (x) = 
\frac{1}{x p_r(x)} \int_{0}^{x} \Big( h(t) - \E \big[ h(X_r) \big] \Big) p_r(t) dt$ for $x >0$. Furthermore, one can extend $f_h^-$ and $f_h^+$ continuously by setting
$f_h^-(0) = f_h^+(0) := \frac{h(0) - \E[h(X_r)]}{r}$. Now, for a given test function $h:\R\to \R$, set $h_1(x) := h(2x-\nu)$. Following \cite[p.~3399]{d-p}, if $f_h$ is the solution of \eqref{eq:GammSteinEq} (with $r=\nu/2$), where $h$ is replaced by $h_1$, then $S(h)(x):= \frac{1}{2} \, f_h \left( \frac{x+\nu}{2} \right)$ solves \eqref{eq:CenteredGamma-stein-eq}. Therefore, the unique bounded solution $S(h)$ of the Stein equation \eqref{eq:CenteredGamma-stein-eq} admits the following explicit representation  
\begin{equation} \label{eq:ExplicitRepresentationOfSolution}
S(h)(x) = \int_{-\nu} ^{x} \bigg( \frac{\hat q(t)}{2 (x+\nu) \hat q(x)} \ind{x \leq -\nu }(x) + \frac{\hat p_{\nu}(t)}{2 (x+\nu) \hat p_{\nu}(x)} \ind{x > - \nu }(x) \bigg)  \Big( h(t) - \E\big[ h(G(\nu)) \big] \Big) \, dt,
\end{equation}
where $\hat p_\nu$ is the density of the centered Gamma distribution $G(\nu)$ given by 
\[ 
\hat p_\nu (x) = \frac{1}{2} \, p_{\nu/2} \left( \frac{x+\nu}{2} \right) =
\begin{cases}
2^{-\frac{\nu}{2}} \,\Gamma \left(\frac{\nu}{2} \right)^{-1} \, (x+ \nu)^{\frac{\nu}{2} -1} \, e^{- \frac{x+ \nu}{2}}, & x > - \nu \\
0, & x \leq -\nu;
\end{cases}
\]
and $\hat q(x) := \frac{1}{2} \, q_l\left( \frac{x+\nu}{2} \right) = - \, 2^{-\frac{\nu}{2}} \big(- (x+\nu) \big)^{\frac{\nu}{2} - 1} \, e^{- \frac{x+\nu}{2}}$. Also note that 
\begin{equation}\label{eq:solution-boundary-value}
S(h)(-\nu) = \frac{h(-\nu) - \E\big[h(G(\nu)) \big]}{\nu}.
\end{equation}

The following lemma will be used in the proof of Proposition \ref{prop:being-compact-operator}. Using a simple adaptation, a similar statement also holds for the solution $S(h)$ corresponding to the Stein equation \eqref{eq:CenteredGamma-stein-eq} of the centered Gamma distribution $G(\nu)$. 
\begin{lem}\label{rem:derivative-bounds}
	Let $X_r \sim \Gamma(r,1)$ with cumulative distribution function $F_r$, and $h$ be a Lipschitz-continuous function. Then there exist two non-negative bounded functions $U^+$ on $(0,+\infty)$, and $U^-$ on $(-\infty,0]$ such that $U^{\pm} \downarrow 0$ as $x\to \pm \infty$, and the following estimates are in order: 
	\begin{enumerate}
		\item[(a)] for $x >0$ it holds that $\Big \vert f'_h (x) \Big \vert \le 2 \Vert h'\Vert_\infty U^+ (x)$,
		\item[(b)] for $x<0$  it holds that $\Big \vert f'_h (x) \Big \vert \le 2 \Vert h'\Vert_\infty U^- (x)$.
	\end{enumerate}
\end{lem}

\begin{proof}
	Let $Q_l(x):= \int_{x}^{0} (- q_l(y))dy$. Consider 
	\begin{align*}
	V^+ (x) := \frac{\int_{0}^{x}  F_r (y) dy   \int_{x}^{\infty}   \left( 1 - F_r (y)\right) dy}{x^2 p_r (x)}, \text{ and }	V^- (x) := \frac{(r-x) \int_{x}^{0} Q_l (y) dy}{-x^2 q_l (x)}.
	\end{align*}
	It is known that both estimates in parts (a) and (b) take place with $V^{\pm}$ instead of $U^{\pm}$ (see  \cite[Corollary 3.15. Part (b)]{Christian_Gamma_Stein}, and \cite[relation (35), page 4304]{d-p}). Moreover, for $x > r$, the function $V^+$ satisfies 
	$$0 \le  V^+ (x) \le U^+ (x):= \frac{ \int_{x}^{\infty}   \left( 1 - F_r (y)\right) dy}{x p_r (x)} \le 1.$$ Also, it is straightforward to check that as $x \to +\infty$, the function $U^+$ is decreasing to $0$. (It is also true that $ 0 \le U^+ (x) \le 1$ for $0 < x \le r$ \cite[see the top of page 3403]{d-p}). Part (b) is similar.
\end{proof}

\subsection{An Operator Theory Approach}\

Let $a,b \in \R^+ \cup \{ \infty \}$. Define
\[
\mathcal{B}_{a,b} := \Big\{ f : \R \to \R, \text{ Lipschitz-continuous} : \norm{f}_{\infty} < a , \text{ and } \norm{f'}_{\infty} < b \Big\}.
\]

\begin{lem}\label{lem:Being-Banach}
Let $\mathcal{B}:= \mathcal{B}_{\infty,\infty}$. For every given  $h \in \mathcal{B}$, define $\norm{f}_{\mathcal{B}}:= \norm{f}_{\infty} + \norm{f'}_{\infty}$. Then $\Vert \cdot \Vert_{\mathcal{B}}$ is a norm on the real vector space $\mathcal{B}$, and furthermore the pair $\left( \mathcal{B}, \Vert \cdot \Vert_{\mathcal{B}}\right)$ is a Banach space, the so-called Lipschitz-space.
\end{lem}
\begin{proof}
It is straightforward to see that the pair $\left( \mathcal{B}, \Vert \cdot \Vert_{\mathcal{B}}\right)$ is a normed space. Furthermore, it is a classical  fact that it is a Banach space, see for example \cite[Proposition 6.1.2]{LipAlg}.
\end{proof}

\begin{lem}\label{lem:being-linear-bounded}
Consider the mapping $S: \mathcal{B} \to \mathcal{B}$ such that for every $h \in \mathcal{B}$, the action $S(h)$ is defined as the unique bounded solution to the centered Gamma Stein equation \eqref{eq:CenteredGamma-stein-eq}, which is guaranteed to exist by Theorem \ref{thm:Gamma-Stein-Solution-Properties} item (a). Then $S(h) \in \mathcal{B}$, and $S$ is a bounded linear operator from the Banach space $\mathcal{B}$ to itself. 	
	\end{lem}

\begin{proof}
Let $h \in \mathcal{B}$. Then a direct application of Theorem \ref{thm:Gamma-Stein-Solution-Properties} item (a) yields that $S(h) \in \mathcal{B}$. To show linearity of $S$, take $h_1,h_2 \in \mathcal{B}$, and $\alpha \in \R$.	Then using the Gamma Stein equation \eqref{eq:CenteredGamma-stein-eq}, together with the fact that $S(h)$ is the unique bounded solution to the latter, we infer that $S(h_1 + \alpha h_2)= S(h_1) + \alpha S(h_2)$. For the boundedness of $S: \mathcal{B} \to \mathcal{B}$ we apply Theorem \ref{thm:Gamma-Stein-Solution-Properties} part (a) to obtain
\begin{align*}
\norm{S(h)}_{\mathcal{B}} = \norm{S(h)}_{\infty} + \norm{S(h)'}_{\infty} \leq \norm{h'}_{\infty} + c_\nu \norm{h'}_{\infty} & \leq (1+ c_\nu) \big( \norm{h}_{\infty} + \norm{h'}_{\infty} \big) \\
& = (1+c_\nu) \norm{h}_{\mathcal{B}}.
\end{align*}
Hence $\Vert S \Vert \le 1+ c_\nu.$
\end{proof}

\begin{prop}\label{prop:no-eigenvalue}
Consider the bounded linear operator $S: \mathcal{B} \to \mathcal{B}$ defined as in Lemma \ref{lem:being-linear-bounded}. Then the following statements are in order.
\begin{enumerate}
 \item[(a)] The operator $S$ does not admit any non-zero eigenvalue, i.e. if $S(h) = \lambda h$ for some non-zero constant $\lambda \in \R$, then necessary $h =0$.
 \item[(b)] For every non-zero scalar $\lambda \in \R$, the operator $I+\lambda S : \mathcal{B} \to \mathcal{B}$ is a one to one map, where $I:\mathcal{B} \to \mathcal{B}$ stands for the identity operator.   
 \end{enumerate}
\end{prop}

\begin{proof}
(a) By contrary assume that there exists a non-zero scalar $\lambda \in \R$ such that
\begin{equation}\label{eq:eigenvalue}
 S(h) = \lambda h.
 \end{equation}
 We claim that $h(-\nu) =0$. Otherwise introduce the auxiliary test function $g = \frac{h}{h(-\nu)} -1$. Then, obviously, $g \in \mathcal{B}$, and moreover by virtue of relation \eqref{eq:eigenvalue}, we have $S(g) = \lambda (g+1)$.  Furthermore, we have $\E \left[  g (G(\nu))\right]=- \lambda \nu$, because $S(g)(-\nu) =\lambda$.  Therefore, the function $g$ satisfies the first order non-homogeneous ode 
 \begin{equation}\label{eq:ode2-no-eigenvalue}
 2 \lambda (x+\nu) g' - (\lambda x + 1) g = \lambda (x+\nu).
 \end{equation}
 Then general solutions of the ode \eqref{eq:ode2-no-eigenvalue} on the interval $(-\nu,\infty)$ are given by
 \begin{equation}\label{eq:ode2-solution+}
 g(x) =  e^{\frac{x}{2}} (x+\nu) ^{\beta} \Big \{ C_3 + \frac{1}{2}  \int_{-\nu}^{x}   e^{-\frac{y}{2}} (y+\nu) ^{-\beta} dy \Big \},
 \end{equation}
 where $\beta:= \frac{1-\lambda \nu}{2 \lambda}$. Now, if $\beta < 1$, then as $x \to +\infty$, we have $$ \int_{-\nu}^{x}  e^{-\frac{y}{2}} (y+\nu) ^{-\beta} dy \to c_\beta < \infty.$$ This implies that $g(x) \to +\infty$ as $x \to +\infty$, which is a contradiction to the fact that $g$ must be a bound function. When $\beta \ge 1$, i.e. $\tilde{\beta}:= 1 - \beta \le 0$ as $x \to +\infty$, we obtain that for some finite constant $d_\beta$ that  
 \begin{equation*}
 \int_{-\nu}^{x}  e^{-\frac{y}{2}} (y+\nu) ^{-\beta} dy \to d_\beta \Gamma(\tilde{\beta}),
 \end{equation*}
 which is either an infinite number or a finite number depending on whether $\tilde{\beta} \in -\N \cup \{0\}$ is a negative integer or not. Therefore, in any case, we have obtained that $g(x) \to +\infty$ as $x \to +\infty$, which is a contradiction. Hence always $h(-\nu) =0$. This implies that $\E \left[ h (G(\nu))\right]=0$ by using \eqref{eq:solution-boundary-value}. On the other hand, $S(h) = \lambda h$ satisfies the first order ode \eqref{eq:CenteredGamma-stein-eq}, and therefore 
\begin{equation}\label{eq:ode1-no-eigenvlaue}
2 \lambda ( x+ \nu) h' - (\lambda x +1) h =0. 
\end{equation} 
The general solutions of the ordinary differential equation \eqref{eq:ode1-no-eigenvlaue} on the interval $(-\nu,\infty)$ are given by 
\begin{equation}\label{eq:ode1-solution+}
h(x) = C_1 e^{\frac{x}{2}} (x+\nu) ^{\frac{1-\nu}{2\lambda}},
\end{equation}
for some constant $C_1$. If $C_1 \neq 0$, then this is a contradiction to the fact that $S(h)$ is a bounded function over the whole real line. Hence it must hold that $C_1=0$. Similarly, the general solutions of the ordinary differential equation \eqref{eq:ode1-no-eigenvlaue} on the interval $(-\infty, -\nu)$ are given by 
\begin{equation}\label{eq:ode1-solution-}
h(x) = C_2 e^{-\frac{x}{2}} (-x-\nu) ^{\frac{1-\nu}{2\lambda}}
\end{equation}
where $C_2$ is a general constant. Now if $C_2 \neq 0$, we infer that $S(h)$ is unbounded on the domain $(-\infty,-\nu)$, which leads to a contradiction. Therefore $C_2 =0$, and as a direct consequence we get $h=0$.

(b) Assume that $\lambda \neq 0$ is a non-zero scalar. Then the mapping $I + \lambda S : \mathcal{B} \to \mathcal{B}$ is a linear operator. Hence, $I + \lambda S$ is a one to one map if and only if  $\Ker\left( I+ \lambda  S \right) = \{0 \}$, and the latter follows at once from part (a).
\end{proof}

\begin{lem}\label{lem:basic-lemma}
Let $ f_n: [a,b] \to \R$ be a sequence of $L$-Lipschitz continuous functions for every $n \in \N$: i.e. for all $x,y \in [a,b]$, and every $n$, $$\Big \vert f_n (x) - f_n (y) \Big \vert \le L \vert x -y \vert.$$ Assume further that $f_n \to f$ pointwise as $n $ tends to infinity. Then $f$ is also an $L$-Lipschitz function and $f_n \to f$ uniformly.	
	\end{lem}
\begin{proof}
It is elementary.
\end{proof}

\begin{prop}\label{prop:being-compact-operator}
The bounded linear operator $S: \mathcal{B} \to \mathcal{B}$ defined as in Lemma \ref{lem:being-linear-bounded} is a compact operator. 
\end{prop}

\begin{proof}
Let $U_{\mathcal{B}}:= \{ h \in \mathcal{B} : \norm{h}_{\mathcal{B}}= \Vert h \Vert_\infty + \Vert h' \Vert_\infty \le  1 \}$ denote the unit ball of the Banach space $\mathcal{B}$. We need to show that the image $S \left( U_{\mathcal{B}} \right)$ of the unit ball is a precompact set in $\mathcal{B}$, or equivalently, that every sequence $\{S(h_n)\}_{n\ge 1} \subseteq S(U_{\mathcal{B}})$ has a convergent subsequence in the topology of the Banach space $\mathcal{B}$. We divide the rest of the proof in three steps.\\
\textit{Step (1)}: First we show that there exists a subsequence $\{ h_{n_k} \}_{k\ge 1}$ such that $h_{n_k} \to h$ pointwise for some $h\in U_{\mathcal{B}}$.  Moreover $S(h_{n_k}) \to S(h)$, and $ S(h_{n_k})'  \to S(h)'$ pointwise. Note that $\{h_n\}_{n\ge 1} \subseteq U_{\mathcal{B}}$ is a bounded subset of $\mathcal{B}$. It is well known (see for example \cite[Chapter 2]{LipAlg} or \cite[Theorem 2.4, and Proposition 2.1]{Nik-1} as well as the survey \cite{Lip-Free-Banach-Spaces-Survey}) that the Banach space $\mathcal{B}$ is a predual space, i.e. there exists a (unique) Banach space $\text{\AE}(\R)$, the so called \textit{Arens-Eells space}, such that $\text{\AE}(\R)^* = \mathcal{B}$.  On the other hand, the Banach-Alaoglu theorem implies that the unit ball $U_{\mathcal{B}}$ is weak-$^*$ compact. Moreover, $\R$ is a separable Banach space, so the Arens-Eells Banach space $\text{\AE}(\R)$ is, too \cite{Lip-Free-Banach-Spaces-Survey}. Hence the  weak-$^*$topology on $U_{\mathcal{B}}$ is metrizable. Therefore, weak-$^*$ compact is the same as weak-$^*$ sequentially compact on the unit ball $U_{\mathcal{B}}$. It follows that the sequence  $\{h_n\}_{n\ge 1}$ contains a subsequence that converges in the weak-$^*$ topology to an element $h \in U_{\mathcal{B}}$. Without loss of generality, we assume that the subsequence is given by the sequence itself. Hence there exists an element $h \in U_{\mathcal{B}}$ such that $h_n \to h$ in the $\text{weak}^*$-topology.  Furthermore, the weak-$^*$ topology on the bounded subsets of $\mathcal{B}$ coincides with the topology of pointwise convergence, see \cite[Proposition 2.1]{Nik-1}. As a consequence, $h_n \to h$ pointwise (here one should not expect that $h_n \to h$ weakly; otherwise this implies that the unit ball is weakly sequentially compact, and therefore the Banach space $\mathcal{B}$ is reflexive which is a contradiction).  An application of the Lebesgue dominated convergence theorem implies that $S(h_n) \to S(h)$ pointwise. Taking into account these observations together with the fact that for every $n \in \N$ we have
\[
2 (x+\nu) S(h_n)'(x) - x S(h_n)(x) = h_n (x) - \E \left[  h_n (G(\nu))\right],
\]
there exists a function $f$ such that $S(h_n)' \to f$ pointwise. On the other hand, for every $x \in \R$ we have that
\[
2 (x+\nu) f(x) = h (x) - \E \left[ h (G(\nu))\right] + x S(h)(x).
\]
Recall that $ h \in U_{\mathcal{B}}$. Hence, the function $S(h)$ satisfies the Gamma Stein equation $$2(x+\nu)S(h)'(x) =  h (x) - \E \left[ h (G(\nu))\right] + x S(h)(x).$$ Hence $f = S(h)'$, and also $S(h_n)' \to S(h)'$ pointwise.\\
\textit{Step (2)}: In this step, we show that $S(U_{\mathcal{B}}) \subseteq C_0(\R)$ is a family of functions having the \textit{equivanishing at infinity} property, i.e. for every given $\epsilon >0$, there exists a compact interval $K \subset \R$ such that $ \big \vert f(x) \big \vert < \epsilon$ for all $f \in  S(U_{\mathcal{B}})$ and for all $x \notin K$. To do this, we use the explicit integral representation \eqref{eq:ExplicitRepresentationOfSolution}. Note that since $\norm{h}_{\infty} \le 1$, we have $\abs{h(t) - \E[h(G_\nu)] } \leq 2$ for all $t\in\R$. When $x > -\nu$, then (recall that $\hat p_{\nu}$ is the density of $G(\nu)$):
\begin{align*}
\Big \vert S(h)(x) \Big \vert & = \abs*{ \int_{-\nu} ^{x} \frac{\hat p_{\nu}(t)}{2 (x+\nu) \hat p_{\nu}(x)} \Big( h(t) - \E\big[ h(G_\nu) \big] \Big) \, dt} \\
& =  \abs*{ \int_{x}^{\infty} \frac{\hat p_{\nu}(t)}{2 (x+\nu) \hat p_{\nu}(x)} \Big( \E\big[ h(G_\nu) \big] - h(t) \Big) \, dt} \\
& \leq \int_{x}^{\infty} \frac{\hat p_{\nu}(t)}{(x+\nu) \hat p_{\nu}(x)} \, dt  = \frac{1}{x+\nu} \int_{x}^{\infty} \left(\frac{t+\nu}{x+\nu}\right)^{\frac{\nu}{2}-1} \frac{e^{-t/2}}{e^{-x/2}} \, dt.
\end{align*}
Now if $\nu \leq 2$, then $ \left(\frac{t+\nu}{x+\nu}\right)^{\nu/2 -1} \leq 1$ and thus
\[ \Big \vert S(h)(x) \Big \vert  \leq  \frac{1}{x+\nu} \int_{x}^{\infty} \frac{e^{-t/2}}{e^{-x/2}} \, dt = \frac{2}{x+\nu} \stackrel{x \to \infty}{\longrightarrow} 0. \]
When $\nu > 2$, set $r:= \ceil{\nu/2 -1}$. We have
\begin{align*}
\Big \vert S(h)(x) \Big \vert & \leq \frac{e^{x/2}}{(x+\nu)^{\nu/2}} \int_{x}^{\infty} (t+\nu)^{\frac{\nu}{2}-1} \, e^{-t/2} \, dt 
 \leq \frac{e^{x/2}}{(x+\nu)^{\nu/2}} \int_{x}^{\infty} (t+\nu)^{r} \, e^{-t/2} \, dt \\
& = e^{\nu/2} \, \frac{e^{x/2}}{(x+\nu)^{\nu/2}} \int_{x+ \nu}^{\infty} t^r \, e^{-t/2} \, dt \\
& = e^{\nu/2} \, \frac{e^{x/2}}{(x+\nu)^{\nu/2}} \, e^{-\frac{x}{2} - \frac{\nu}{2}} \sum_{i=0}^{r} (-1)^{r-i+1} \frac{r!}{i! (-\frac{1}{2})^{r-i+1}} (x+\nu)^i \\
& =: \frac{P(x)}{(x+\nu)^{\nu/2}},
\end{align*}
where $P$ is a polynomial of degree $r$. Since we always have $r < \nu/2$, it follows that $\lim_{x\to \infty} \abs{S(h)(x)} = 0 $. When $x < - \nu$, again using \eqref{eq:ExplicitRepresentationOfSolution} of the explicit representation of the solution function $S(h)$, we get  
\begin{align*}
\abs{S(h)(x)} & = \abs*{ \int_{-\nu} ^{x} \frac{\hat q_{\nu}(t)}{2 (x+\nu) \hat q_{\nu}(x)} \Big( h(t) - \E\big[ h(G_\nu) \big] \Big) \, dt} \\
& \leq  \int_{x}^{-\nu} \frac{\hat q_{\nu}(t)}{(-x - \nu) \hat q_{\nu}(x)} \, dt  = \frac{1}{-x -\nu} \int_{x}^{-\nu} \left(\frac{-t-\nu}{-x-\nu}\right)^{\frac{\nu}{2}-1} \frac{e^{-t/2}}{e^{-x/2}} \, dt.
\end{align*}
Hence, the case $x \to - \infty$ can now be discussed similarly. Note that the upper bounds for $\abs{S(h)(x)}$ that we found do not depend on the choice of the test function $h$. Therefore, we have shown that, in addition to $S(U_{\mathcal{B}}) \subseteq C_0(\R)$, the collection $S(U_{\mathcal{B}})$ is a family of functions that are equivanishing at infinity. \\
\textit{Step (3)}: Next we show that as $n \to \infty$,
\begin{equation}\label{eq:uniform-conv}
\Big \Vert S(h_n) - S(h) \Big\Vert_{\mathcal{B}}= \Big \Vert S(h_n) - S(h) \Big\Vert_\infty  +  \Big \Vert S(h_n)' - S(h)' \Big \Vert_\infty \to 0.
\end{equation}
By Step $(2)$, for a given $\epsilon>0$, there exists a compact interval $K \subset \R$ such that 
\begin{equation}\label{eq:equi-vanish}
\sup_{n\ge 1} \sup_{x \notin K} \max \Big\{ \Big \vert S(h_n) (x)\Big \vert, \Big \vert S(h)(x) \Big \vert \Big\} < \epsilon.
\end{equation}
 On the other hand, the family $(S(h_n): n\ge 1)$ consists of $1$-Lipschitz-continuous functions (see part (a), Theorem  \ref{thm:Gamma-Stein-Solution-Properties}), and by step (1) converges pointwise to $S(h)$ on the compact interval $K$. Hence, Lemma \ref{lem:basic-lemma} yields that  
 \begin{equation}\label{eq:unifrom-solution-1}
 S(h_n) \to S(h) \quad  \text{uniformly on } \, K.
 \end{equation}
Finally relations \eqref{eq:equi-vanish} and \eqref{eq:unifrom-solution-1} readily imply that $S(h_n) \to S(h)$ uniformly on the real line. Now, we are left to show that $\Vert S(h_n)' - S(h)' \Vert_\infty \to 0$. To this end, first note that for every $h \in U_{\mathcal{B}}$, and every $x \neq y \in \R$ it holds that $\vert S(h)' (x) - S(h)'(y) \vert \le c_\nu \Vert h'\Vert_\infty \vert x -y \vert \le c_\nu \vert x -y \vert$. Hence, the family $\{ S(h_n)', S(h)' \, : n\ge 1 \}$ consists of $c_\nu$-Lipschitz continuous functions. On the other hand, Lemma \ref{rem:derivative-bounds} yields that the family $\{ S(h_n)', S(h)' \, : n\ge 1 \}$ is equivanishing at infinity. The result now follows.

\if 

====================================
First we show that the set $S \left( U_{\mathcal{B}} \right)$ is precompact in the space of continuous functions defined over the whole real line and vanishing at infinity equipped with the uniform topology, i.e. $$C_0(\R) = \Big \{ h: \R \to \R, \text{ continuous } : \lim_{\vert x \vert \to +\infty} h(x) =0 \Big \}.$$ The proof is rather lengthly, and it is divided into several steps in below.\\

\textit{Step (2)}: The family $S(U_{\mathcal{B}})$ is \textit{equicontinuous} in the sense that for all $, \epsilon > 0, x \in \R$, there exists an open neighborhood $x \in U_x$ such that for all $y \in U_x$, and all $f \in S(U_{\mathcal{B}})$: $ \vert f(y) - f(x) \vert < \epsilon$.  To this end, take an arbitrary $f = S(h) \in S(U_{\mathcal{B}})$ for some $h \in U_{\mathcal{B}}$. Then, according to Theorem \ref{thm:Gamma-Stein-Solution-Properties} part (a), $f$ is a Lipschitz-continuous function satisfying $\vert f(x) - f(y) \vert \le C_\nu \Vert h'\Vert_\infty \vert x -y \vert  \le C_\nu \vert x -y \vert$. Note that $\Vert h' \Vert _\infty \le \Vert h \Vert_{\mathcal{B}} \le 1$. Now the claim follows at once. \\
\textit{Step (3)}: The family $S(U_{\mathcal{B}})$ is \textit{pointwise bounded} in the sense that $\sup_{h \in U_{\mathcal{B}}} \vert S(h) (x) \vert < \infty$ for all $x \in \R$. This also follows from Theorem \ref{thm:Gamma-Stein-Solution-Properties} part (a) that $$\sup_{h \in U_{\mathcal{B}} }\sup_{x \in \R} \vert S(h) (x)\vert \le  \sup_{h \in U_{\mathcal{B}} } \Vert h' \Vert \le \sup_{h \in U_{\mathcal{B}} } \Vert h \Vert_{\mathcal{B}} \le   1.$$
\textit{Step (4)}: The previous steps together with a direct application of Theorem 5 \cite{arzela-ascoli} yield that every sequence $\{S(h_n)\}_{n \ge 1} \subseteq S(U_{\mathcal{B}})$ has a uniform convergence subsequence. Without loose of generality, we assume that the subsequence  is the sequence itself, and therefore as $n \to \infty$, we get $\Vert S(h_n) - f \Vert_\infty \to 0$ for some $f \in C_0(\R)$. Now, let $x \neq y \in \R$. Then
\begin{eqnarray}
\Big \vert 	\frac{f(x) -f (y)}{x-y} \Big \vert = \lim_{n\to \infty} \Big \vert 	\frac{S(h_n)(x) -S(h_n) (y)}{x-y} \Big \vert \le \limsup_{n\to \infty} \Vert S(h_n)'\Vert_\infty \\\le C_\nu \limsup_{n\to \infty} \Vert h'_n \Vert_\infty \le C_\nu  \limsup_{n\to \infty}  \Vert h_n\Vert_{\mathcal{B}} \le C_\nu.
\end{eqnarray}
Hence, $f$ is a Lipschitz-continuous function with $\Vert f' \Vert_\infty \le C_\nu$. This implies that $f \in \mathcal{B}$. \\
\textit{Step (5)}: Next, we claim that $\Vert S(h_n) - f \Vert_\mathcal{B} \to 0$  as $n \to \infty$. To this end, using the linearity of $S$, since $\Vert S(h_n - h_m)\Vert_\infty \to 0$ as $n,m \to \infty$, therefore, it is enough to show that for every $x \neq y \in \R$, as $n,m \to \infty$, we have 
\begin{equation}\label{eq:convergence-Lipschitz-norm}
\frac{\Big \vert  S(h_n - h_m) (x) - S(h_n -h_m)(y)  \Big \vert }{\vert x -y \vert} \to 0.
\end{equation}
Let $ -\nu < y \le x$. Then
\begin{equation*}
\begin{split}
 S(h_n - h_m) (x) &- S(h_n -h_m)(y)  
 =\Bigg\{ \int_{-\nu}^{y} \Big[ (h_n - h_m)(t) - \E\left[  (h_n - h_m)(G(\nu))\right]\Big] \hat p_\nu (t) dt \Bigg\} \\ & \qquad \times \Bigg\{  \frac{1}{2(x+\nu) \hat p_\nu (x)} - \frac{1}{2(y+\nu) \hat p_\nu (y)} \Bigg\}\\
 &+ \frac{1}{2(x+\nu) \hat p_\nu (x)} \int_{y}^{x} \Big[ (h_n - h_m)(t) - \E\left[  (h_n - h_m)(G(\nu))\right]\Big] \hat p_\nu (t) dt \\
 &:= I_1 + I_2.
 \end{split}
\end{equation*}
Now, note that $$I_1 = \Bigg\{ \frac{(y+\nu) \hat p_\nu (y) - (x+\nu) \hat p_\nu (x)}{(x+\nu) \hat p_\nu (x)} \Bigg\} \times S(h_n-h_m)(y).$$ Hence, we are left to show that there exists a constant $M_1>0$ such that $$\Bigg \vert 1 - \frac{(y+\nu)}{(x+\nu)} \frac{\hat p_\nu (y)}{\hat p_\nu (x)} \Bigg \vert \le M_1 \vert x - y \vert.$$
{\color{blue} We will come back to check out this issue later on. But how about $I_2$ term ?}

\fi

\end{proof}

\begin{thm}\label{thm:fredholm-alternative}
Let $\lambda \in \R$ be a non-zero scalar. Then for every $h\in \mathcal{B}$ there exists a unique solution $g \in \mathcal{B}$ to the functional equation 
\begin{equation}\label{eq:functional-eq}
h = \left( I + \lambda S \right)(g) = g + \lambda S(g).
\end{equation}
\end{thm}
\begin{proof}
This is a direct application of Propositions \ref{prop:no-eigenvalue}, \ref{prop:being-compact-operator}, and the classical Fredholm alternative Theorem  \cite[3.4.24, page 329]{Banach-Space-Theory}.	
	\end{proof}

For $r > 0$, let $U_{\mathcal{B}} (r): = \{ h \in \mathcal{B} :  \Vert h \Vert_{\mathcal{B}} \le r \}$ denote the ball of radius $r$.
\begin{prop}\label{prop:universal-bound-solution}
Let $r_1 >0$, and $\lambda \in \R$ be a non-zero scalar. Then there exists a universal constant $r_2$ (may depend on $r_1$, $\lambda$, and $\nu$) such that for every $h \in U_{\mathcal{B}} (r_1)$ the unique solution $g$ of the functional equation \eqref{eq:functional-eq} satisfies $\Vert g \Vert_{\mathcal{B}} \le r_2$. 
\end{prop}

\begin{proof}
From Proposition \ref{prop:no-eigenvalue} and Theorem \ref{thm:fredholm-alternative}, the linear bounded operator $I + \lambda S : \mathcal{B} \to \mathcal{B}$ is a bijective map. Hence the result follows at once using the inverse mapping Theorem \cite[1.6.6 Corollary]{Banach-Space-Theory}.
\end{proof}

\section{Optimal Gamma Approximation}\label{sec:optimal}

\subsection{A General Stein-Malliavin Upper Bound} \label{sec:MainUpperBound}
In the following, we present a general Malliavin-Stein upper bound that constitutes the cornerstone to achieve our final optimal goal.  We start with the following useful result. Sometimes, we will use centered versions of the Gamma-operators, i.e.
\[  \CenteredGamma_j(F)  := \Gamma_j(F) - \E[\Gamma_j(F)] .\]

\begin{Prop}\label{prop:first-general-upper-bound}
	Let $F$ be a centered random variable admitting a finite chaos expansion with $\Var(F) = 2 \nu$. Let $G(\nu) \sim CenteredGamma(\nu)$. Then there exists a constant $C>0$ (only depending on $\nu$), such that
	\begin{equation}\label{first-main-theoretical-estimate} 
	d_2(F,G(\nu)) \leq C  \sup_{h\in \mathcal{B}_{1,1}} \E \, \Big \vert h(F)  \big(  \CenteredGamma_1(F) - 2 F \big) \Big \vert,
	\end{equation}
	where recall that $\mathcal{B}_{1,1}:= \big \{ h: \R \to \R,  \, \text{Lipschitz-continuous} \, : \, \Vert h \Vert \le 1, \, \Vert h' \Vert_\infty \le 1 \big \}$.
\end{Prop}

\begin{proof}
	Consider the centered Gamma Stein equation \eqref{eq:CenteredGamma-stein-eq}. Let $h \in \mathcal{H}_2$ be an arbitrary test function (note that $\E \vert h(G(\nu)) \vert < \infty$). Then by using the Malliavin integration by parts formula \eqref{eq:IntegrationByParts}, we get
	\begin{align*}
	\abs*{\E[h(F)] - \E[h(G(\nu))] } & = \abs*{ \E \left[ 2(F+ \nu) \, S(h)'(F) - F S(h)(F) \right] } \\
	& = \abs*{ \E \left[ 2(F+ \nu) \, S(h)'(F) - S(h)'(F) \sprod{DF, - D L^{-1} F }_{\mathfrak{H}} \right] } \\
	& = \abs*{ \E \left[ S(h)'(F) \left( \CenteredGamma_1(F) - 2 F\right)\right] }.
	\end{align*}
	Now the claim follows at once by a direct application of Theorem \ref{thm:Gamma-Stein-Solution-Properties}.
\end{proof}

To simplify computations, we continue with the following useful Lemmas.

\begin{Lem} \label{lem:AuxiliarlyLemmaForMainTheorem}
	Let $g : \R \to \R$ be a Lipschitz-continuous function, where $g$ and $g'$ are bounded by a constant only depending on $\nu > 0$. Consider the solution $S(g)$ of the Gamma Stein equation \eqref{eq:CenteredGamma-stein-eq} associated to the test functions $g$.  Assume that $F\in \mathbb{D}^{\infty}$ is a centered random variable with variance $\E[F^2]= 2 \nu$.  Then for any $r \in \N$:
	\begin{align*}
	\E \Big[ g(F) \Big( \CenteredGamma_r(F) - 2 \CenteredGamma_{r-1}(F) \Big) \Big] = & - \E\Big[ S(g)'(F) \Big( \CenteredGamma_r(F) - 2 \CenteredGamma_{r-1}(F) \Big) \Big( \CenteredGamma_1(F) - 2 F \Big) \Big] \\
	& - \E \Big[ S(g)(F) \Big( \Gamma_{r+1}(F) - 2 \Gamma_r(F) \Big) \Big].
	\end{align*}
\end{Lem}

\begin{proof}
	First note that $2\nu = \E[ \Gamma_1(F)]$. Thus
	\begin{align*}
	\E \Big[ g(F) \Big( \CenteredGamma_r(F) - 2 \CenteredGamma_{r-1}(F) \Big) \Big] & = \E \Big[ \Big( g(F) - \E\big[g(G(\nu)) \big] \Big)  \Big( \CenteredGamma_r(F) - 2 \CenteredGamma_{r-1}(F) \Big) \Big] \\
	& = \E \Big[ \Big( 2(F+\nu) S(g)'(F) - F S(g)(F) \Big) \Big( \CenteredGamma_r(F) - 2 \CenteredGamma_{r-1}(F) \Big) \Big]  \\
	& \hspace{-10em} = 
	2 \, \E\big[ F S(g)'(F) \CenteredGamma_r(F) \big]
	+ \E \big[ \Gamma_1(F) \big] \E \big[S(g)'(F) \CenteredGamma_r(F) \big]
	- \E \big[ F S(g)(F) \CenteredGamma_r(F) \big] \\
	& \hspace{-9em} 
	- 4 \, \E\big[ F S(g)'(F) \CenteredGamma_{r-1}(F) \big]
	- 2 \, \E\big[ \Gamma_1(F) \big] \E \big[ S(g)'(F) \CenteredGamma_{r-1}(F) \big]
	+ 2 \, \E \big[ F S(g)(F) \CenteredGamma_{r-1}(F) \big] \\
	& = : \sum_{i=1}^{6} T_i.
	\end{align*}
	Now, we use the integration-by-parts formula \eqref{eq:IntegrationByParts}  in combination with the chain rule \eqref{eq:ChainRule} to obtain
	\begin{align*}
	T_3 + T_2 & = - \E \big[ F S(g)(F) \CenteredGamma_r(F) \big] + \E \big[ \Gamma_1(F) \big] \E \big[S(g)'(F) \CenteredGamma_r(F) \big] \\
	& = - \E \big[ \Gamma_1(F) \CenteredGamma_r(F) S(g)'(F) \big] - \E \big[ S(g)(F) \Gamma_{r+1}(F) \big] + \E \big[ \Gamma_1(F) \big] \E \big[S(g)'(F) \CenteredGamma_r(F) \big] \\
	& = - \E \big[ \CenteredGamma_1(F) \CenteredGamma_r(F) S(g)'(F) \big] - \E \big[ S(g)(F) \Gamma_{r+1}(F) \big],
	\end{align*}
	and similarly
	\[
	T_6 + T_5 = 2 \, \E \big[ \CenteredGamma_1(F) \CenteredGamma_{r-1}(F) S(g)'(F) \big] + 2 \, \E \big[ S(g)(F) \Gamma_r(F) \big].
	\]
	Hence, putting everything together, the result follows.
\end{proof}

\begin{lem}\label{lem:komaki2}
Let $g : \R \to \R$ be a Lipschitz-continuous function, where $g$ and $g'$ are bounded by a constant only depending on $\nu > 0$. Assume that $S(g)$ and $S\left( S(g)\right)$ stand for the solutions of the Gamma Stein equation \eqref{eq:CenteredGamma-stein-eq} associated to the test functions $g$ and $S(g)$ respectively. Let $F\in \mathbb{D}^{\infty}$ be a centered random variable with variance $\E[F^2]= 2 \nu$. Then the following identities take place.
\begin{itemize}
\item[(a)]
\begin{align*}
\E & \Big[ g(F) \Big( 2 F - \CenteredGamma_{1}(F) \Big) \Big] {}={} \E\Big[ S(g)'(F) \Big( \CenteredGamma_1(F) - 2 F \Big)^2 \Big] + \E \Big[ S(g)(F) \Big( \CenteredGamma_{2}(F) - 2 \CenteredGamma_1(F) \Big) \Big] \\
& \quad  - \Big( \E [S(g)(F)]\Big) \Big( \frac{1}{2} \kappa_3(F) - 2 \kappa_2(F) \Big) \\ 
\end{align*}
\item[(b)]
\begin{align*}
\E & \Big[ g(F) \Big( 2 F - \CenteredGamma_{1}(F) \Big) \Big] \\
& {}={} \E\Big[ S(g)'(F) \Big( \CenteredGamma_1(F) - 2 F \Big)^2 \Big] - \E\Big[ S \left( S(g)\right)'(F) \Big( \CenteredGamma_2(F) - 2 \CenteredGamma_{1}(F) \Big) \Big( \CenteredGamma_1(F) - 2 F \Big) \Big] \\
& \quad - \E \Big[ S\left( S(g)\right)(F) \Big( \CenteredGamma_{3}(F) - 2 \CenteredGamma_2(F) \Big) \Big] - \Big(\E [S(g)(F)] \Big)\Big( \frac{1}{2} \kappa_3(F) - 2 \kappa_2(F) \Big) \\
& \quad - \E \big[ S\left( S(g)\right)(F) \big] \Var\big( \Gamma_1(F) - 2 F \big) + \E \big[ S\left( S(g)\right)(F) \big] \Big(  \frac{1}{3} \kappa_4(F) - 3 \kappa_3(F) + 4 \kappa_2(F) \Big).
\end{align*}
\end{itemize}
\end{lem}
\begin{proof}
	 We apply Lemma \ref{lem:AuxiliarlyLemmaForMainTheorem} twice to obtain 
\begin{align*}
\E & \Big[ g(F) \Big( 2 F - \CenteredGamma_{1}(F) \Big) \Big] = \E\Big[ S(g)'(F) \Big( \CenteredGamma_1(F) - 2 F \Big)^2 \Big] + \E \Big[ S(g)(F) \Big( \Gamma_{2}(F) - 2 \Gamma_1(F) \Big) \Big] \\
& {}={} \E\Big[ S(g)'(F) \Big( \CenteredGamma_1(F) - 2 F \Big)^2 \Big] - \Big( \E [S(g)(F)]\Big) \Big( \E \left[ \Gamma_{2}(F)\right] - 2 \kappa_2(F) \Big) \\
& \quad + \E \Big[ S(g)(F) \Big( \CenteredGamma_{2}(F) - 2 \CenteredGamma_1(F) \Big) \Big] \qquad 
   \Big(  \text{this completes the proof of part (a)}\Big)\\
&  {}={} \E\Big[ S(g)'(F) \Big( \CenteredGamma_1(F) - 2 F \Big)^2 \Big] - \Big( \E [S(g)(F)]\Big) \Big( \E \left[ \Gamma_{2}(F)\right] - 2 \kappa_2(F) \Big) \\
& \quad - \E\Big[ S \left( S(g)\right)'(F) \Big( \CenteredGamma_2(F) - 2 \CenteredGamma_{1}(F) \Big) \Big( \CenteredGamma_1(F) - 2 F \Big) \Big] - \E \Big[ S \left(S(g)\right)(F) \Big( \Gamma_{3}(F) - 2 \Gamma_2(F) \Big) \Big] \\
& {}={} \E\Big[ S(g)'(F) \Big( \CenteredGamma_1(F) - 2 F \Big)^2 \Big] - \E\Big[ S \left( S(g)\right)'(F) \Big( \CenteredGamma_2(F) - 2 \CenteredGamma_{1}(F) \Big) \Big( \CenteredGamma_1(F) - 2 F \Big) \Big]\\
& \quad - \E \Big[ S \left(S(g)\right)(F) \Big( \CenteredGamma_{3}(F) - 2 \CenteredGamma_2(F) \Big) \Big] - \Big( \E [S(g)(F)]\Big) \Big( \frac{1}{2} \kappa_3(F) - 2 \kappa_2(F) \Big) \\
& \quad + \Big( \E \big[ S \left( S(g)\right)(F) \big]\Big) \Big( \E \big[\Gamma_3(F) \big] - \kappa_3(F) \Big).
\end{align*}
Note that we cannot translate $\E[\Gamma_3(F)]$ directly into the fourth cumulant, but instead by Proposition \ref{Prop:RelationOldAndNewGamma} part \textit{(c)}, we have $\E[\Gamma_3(F)] = \frac{1}{3} \kappa_4(F) - \Var(\Gamma_1(F))$. The variance term can be written as
\begin{align*}
\Var\big( \Gamma_1(F) \big) & = \Var\big( \Gamma_1(F) - 2 F \big) - 4 \kappa_2(F) + 4 \E\big[ F \Gamma_1(F) \big] \\
& = \Var\big( \Gamma_1(F) - 2 F \big) - 4 \kappa_2(F) + 4 \E\big[ \Gamma_2(F) \big] \\
& = \Var\big( \Gamma_1(F) - 2 F \big) - 4 \kappa_2(F) + 2 \kappa_3(F).
\end{align*}
Putting everything together, the claim follows.
\end{proof}
\begin{rem}\label{rem:cumulants-vanish}
We point out that for both linear cumulant combinations appearing in the right hand sides of parts (a) and (b) in Lemma \ref{lem:komaki2} it holds that 
\begin{align*}
 \frac{1}{2} \kappa_3(G(\nu)) - 2 \kappa_2(G(\nu)) =  0, \quad \text{and} \quad 
 \frac{1}{3} \kappa_4(G(\nu)) - 3 \kappa_3(G(\nu)) + 4 \kappa_2(G(\nu))  = 0.
 \end{align*}
	\end{rem}

Now, we are ready to state the main result of this section.

\begin{Thm}\label{thm:MainMalliavinSteinBound}
	Let $F$ be a centered random variable admitting a finite chaos expansion with $\Var(F) = 2 \nu$. Let $G(\nu) \sim CenteredGamma(\nu)$. Then there exists a constant $C>0$ (only depending on $\nu$), such that
	\begin{align}\label{main-theoretical-estimate}
	d_2(F,G(\nu)) \leq C \, \bigg\{ & \Var \left(  \Gamma_1(F) - 2F \right) \notag  + \sqrt{\Var \left(  \Gamma_2(F) - 2 \Gamma_1(F) \right) } \times \sqrt{ \Var  \left(  \Gamma_1(F) -2F \right)} \notag \\
	& + \sqrt{ \Var \Big(  \big( \Gamma_3(F) - 2 \Gamma_2(F) \big)  - 2 \big( \Gamma_2(F) - 2 \Gamma_1(F) \big)\Big) } \notag \\
	& + \Big \vert \kappa_3(F) - \kappa_3(G(\nu)) \Big \vert  +  \Big \vert \kappa_4(F) - \kappa_4(G(\nu))  \Big \vert \bigg\}.
	\end{align}
\end{Thm}

\begin{proof}
 Using Proposition   \ref{prop:first-general-upper-bound}, Theorem \ref{thm:fredholm-alternative} with $\lambda=2$, and Proposition \ref{prop:universal-bound-solution} we obtain that 
 \begin{equation*}
 \begin{split}
 	d_2(F,G(\nu)) &\leq C  \sup_{h\in \mathcal{B}_{1,1}} \E \, \Big \vert h(F)  \big(  \CenteredGamma_1(F) - 2 F \big) \Big \vert\\
 	& \leq  C \sup_{h\in \mathcal{B}_{1,1}} \E \, \Big \vert \big( h(F) + 2 S(h)(F) \big)  \big(  \CenteredGamma_1(F) - 2 F \big) \Big \vert\\
 	\end{split}
 \end{equation*}
 where $C$ stands for a general constant depending only on the parameter $\nu$. Now, we apply Lemma \ref{lem:komaki2} item (b) on $ \E \left[ h(F) \big(  \CenteredGamma_1(F) - 2 F \big)  \right]$, and item (a) on $\E \left[ S(h)(F) \big(  \CenteredGamma_1(F) - 2 F \big) \right]$. Then putting everything together the result follows by applying Cauchy-Schwarz inequality, Theorem \ref{thm:Gamma-Stein-Solution-Properties}, as well as using the fact that $\kappa_2(G(\nu)) = \kappa_2(F) = 2 \nu$, $\kappa_3(G(\nu))= 8\nu$ and $\kappa_4(G(\nu)) = 48 \nu$, see \eqref{eq:CenteredGammaCumulants}.
\end{proof}

\begin{rem}\label{rem:Operator-Theory-Essential}
The splitting technique implemented in the proof of Theorem \ref{thm:MainMalliavinSteinBound} by using operator theory is vital to obtain an optimal upper bound. In fact, not doing it, instead of estimate \eqref{main-theoretical-estimate}, the best estimate one can achieve (under the assumption in Theorem \ref{thm:MainMalliavinSteinBound}) is a similar bound as \eqref{main-theoretical-estimate} with the quantity $\sqrt{ \Var   \left( \Gamma_3(F) - 2 \Gamma_2(F) \right) }$ instead of
\[
\sqrt{ \Var \Big(  \big( \Gamma_3(F) - 2 \Gamma_2(F) \big)  - 2 \big( \Gamma_2(F) - 2 \Gamma_1(F) \big)\Big) }.
\]
On the other hand, it is not difficult to see that for a sequence $\{ F_n= \sum_{1 \le i \le \nu} c_{i,n} (N^2_i -1): n \ge 1\}$ in the second Wiener chaos with a finite number of non-zero spectral coefficients such that for every $i=1,\ldots,\nu$, $c_{i,n} \to 1$ as $n \to \infty$ it holds that
\[
\Var   \left( \Gamma_3(F_n) - 2 \Gamma_2(F_n) \right) \approx_C \Var   \left( \Gamma_2(F_n) - 2 \Gamma_1(F_n) \right) \approx_C \Var   \left( \Gamma_1(F_n) - 2 F_n \right),
\]
resulting in a suboptimal rate. See also illustrating Example \ref{ex:naive-example} for further clarifications.

	\end{rem}

\subsection{The Upper Bound: Second Wiener Chaos}
In the present section, in order to handle the variance quantities of the Gamma operators appearing in the right hand side of estimate \eqref{main-theoretical-estimate} in terms of cumulants, we consider the case of second Wiener chaos. In this setting, the connection is apparent thanks to Lemma \ref{lem:Var(Gamma_r-2Gamma_r-1)}.
\begin{prop}\label{prop:Var-estimate-1}
	Let $\nu >0$, and $F= I_2 (f)$ be in the second Wiener chaos such that $\E[F^2]=2\nu$. Then, for every $r\ge 1$, with constant $C=4\nu$, we have 
	\begin{equation}\label{eq:1}
	\Var \left( \Gamma_{r+1} (F) - 2 \Gamma_r (F) \right) \le \, C \Var \left( \Gamma_r (F) - 2 \Gamma_{r-1}(F) \right) \le \, C^{r} \, \Var \left( \Gamma_1 (F) - 2 F \right).
	\end{equation}
	In particular, by choosing $r=1$, we obtain 
	\begin{equation}\label{eq:optimal-Var1}
	\Var \left( \Gamma_{2} (F) - 2 \Gamma_1 (F) \right) \le (4\nu) \, \Var \left( \Gamma_1 (F) - 2 F \right).
	\end{equation}
\end{prop}

\begin{proof}
Let's prove the first estimate in \eqref{eq:1}. Then the second estimate could be proven by iteration using similar arguments. Let $r \ge 1$. Denote by $A_f$ the associated Hilbert-Schmidt operator. As in the proof of Lemma \ref{lem:Var(Gamma_r-2Gamma_r-1)}, we can write  
\begin{align*}
\Var \left( \Gamma_{r+1} (F) - 2 \Gamma_r (F) \right) & = 2^{2r+3} \Tr \left(   (A^{r+2}_f  - A^{r+1}_f )^2 \right)
 =  2^{2r+3} \Tr \left(   A^2_f (A^{r+1}_f  - A^{r}_f )^2 \right)\\
& \le  2^{2r+3} \Tr (A^2_f)  \times \Tr \left(   (A^{r+1}_f  - A^{r}_f )^2 \right)\\
&= 4 \nu \, \Var \left( \Gamma_{r} (F) - 2 \Gamma_{r-1} (F) \right),
\end{align*}
where in the third step, we have used the trace inequality $\Tr (AB) \le \Tr(A) \, \Tr(B)$ for non-negative operators $A ,B \ge 0$, see \cite{Liu}.
\end{proof}

\begin{rem}
The estimates in \eqref{eq:1} can also deduce from representation \eqref{eq:CentredGammaInTermsOfContraction} together with  the classical estimate $(4.4)$ in \cite[Lemma 4.2]{OptBerryEsseenRates}.
	\end{rem}

\begin{prop}\label{prop:Var-estimate-2}
	Let $\nu >0$, and $F= I_2 (f)$ in the second Wiener chaos such that $\E[F^2]=2\nu$. Assume $r\ge 1$. Then there exists a general constant $C$ (possibly depending on the parameters $\nu$ and $r$) such that
	\begin{align*}
	\Var & \Big( \left(  \Gamma_{2r+1}(F) - 2 \Gamma_{2r}(F) \right) -2 \left(  \Gamma_{2r}(F) - 2 \Gamma_{2r-1}(F) \right)  \Big)  \le 2\,  \Var^{\, 2} \left(  \Gamma_{r}(F) -2 \Gamma_{r-1} (F) \right)\\
	&  \qquad \le_C \Var \left( \Gamma_{r-1}(F) - 2 \Gamma_{r-2}(F) \right) \times \Var \left( \Gamma_{r+1}(F) - 2 \Gamma_{r}(F) \right).
	\end{align*}
	In particular, by choosing $r=1$, we obtain the crucial estimate  
		\begin{equation}\label{eq:optimal-Var2}
	\Var  \Big( \left(  \Gamma_{3}(F) - 2 \Gamma_{2}(F) \right) -2 \left(  \Gamma_{2}(F) - 2 F \right)  \Big)  \le 2 \Var^{\, 2} \left(  \Gamma_{1}(F) -2 F \right).
	\end{equation}	
\end{prop}

\begin{proof}
For the first estimate, using representation \eqref{eq:CentredGammaInTermsOfContraction} we can write
\begin{multline*}
\Var  \Big( \left(  \Gamma_{2r+1}(F) - 2 \Gamma_{2r}(F) \right) -2 \left(  \Gamma_{2r}(F) - 2 \Gamma_{2r-1}(F) \right)  \Big)  \\= 2^{4r+3} \Big \Vert f \contIterated{1}{2r+2} f - 2 f \contIterated{1}{2r+1} f + f \contIterated{1}{2r} f \Big \Vert^2_{\HH^{\otimes 2}}\\
= 2^{4r+3} \Big \Vert  \left(    f \contIterated{1}{r+1} f -  f \contIterated{1}{r} f \right)  \otimes_1 \left(    f \contIterated{1}{r+1} f -  f \contIterated{1}{r} f \right)  \Big \Vert^2_{\HH^{\otimes 2}}\\
\le 2^{4r+3} \Big \Vert  f \contIterated{1}{r+1} f -  f \contIterated{1}{r} f  \Big \Vert^4_{\HH^{\otimes 2}}\\
=  2 \,  \Var^{\, 2} \left(  \Gamma_{r}(F) -2 \Gamma_{r-1} (F) \right),
\end{multline*}
where we have used the classical estimate $(4.4)$ in \cite[Lemma 4.2]{OptBerryEsseenRates}. The second estimate is a direct application of \cite[Corollary 1]{D-trace-1} with $P=(A^{r+1}_f - A^{r}_f)^2, C=A^2_f$ combined with $\Var\left( \Gamma_{r+1}(F) - 2 \Gamma_{r}(F) \right) = 2^{2r+3} \Tr \left( (A^{r+2}_f - A^{r+1}_f)^2 \right)$ for  every $r \ge 0$, see the proof of Lemma \ref{lem:Var(Gamma_r-2Gamma_r-1)}.
\end{proof}

	\subsection{The Lower Bound: Second Wiener Chaos}
	\begin{prop}\label{prop:lower-bound}
	Let $\nu >0$, and $F= I_2 (f)$ be in the second Wiener chaos such that $\E[F^2]=2\nu$. Then there exists a general constant $C$ (possibly depending on the parameter $\nu$) such that 
	\begin{equation*}
	 d_{2} (F, G(\nu)) \ge_C \,  \mathbf{M} (F),
	\end{equation*}
	where the quantity $\mathbf{M} (F)$ is given by \eqref{eq:M}.
\end{prop}

\begin{proof}
 Fix a real number $\rho >0$ whose range of values will be determined later on. Taking into account the second moment assumption, it is a classical result (see \cite[Chapter $7$]{lukas}) that the characteristic functions $\phi_{F}$ and $\phi_{G(\nu)}$ are analytic inside the strip $\Delta_\nu := \{ z \in \mathbb{C} :  \abs{ \operatorname{Im}z } < \frac{1}{2\sqrt{\nu}} \}$. Moreover, in the strip of regularity $\Delta_\nu$, they follow the integral representations
\[
\phi_{F}(z) = \int_\R e^{izx} \mu(dx) \quad \text{and} \quad \phi_{G(\nu)}(z) = \int_\R e^{izx} \mu_\nu (dx),
\]
where $\mu$ and $\mu_\nu$ stand for the probability measures of $F$ and $G(\nu)$ respectively. Recall that all elements in the second Wiener chaos have exponential moments, see \cite[Proposition $2.7.13$, item (iii)]{n-pe-1}. Denote by $\Omega_{\rho, \nu}$ the domain
\[\Omega_{\rho, \nu}: = \Big\{ z=t + i y \in \mathbb{C} \, : \, \abs{ \operatorname{Re}z } < \rho, \abs{ \operatorname{Im}z } < \min \{ (2\sqrt{\nu})^{-1}, e^{-1}\} \Big\}.
\]
Then for any $z \in \Omega_{\rho,\nu}$, together with a Fubini's argument, we have that
\begin{align*}
\abs[\Big]{ \phi_{F}(z) - \phi_{G(\nu)}(z) } & = \abs[\Big]{ \int_\R  e^{itx - y x} (\mu - \mu_\nu) (dx) } = \abs[\Big]{ \sum_{ k \ge 0} \frac{(-y)^k}{k!} \int_\R x^k e^{itx} (\mu - \mu_\nu)(dx) } \\
& \le \sum_{k\ge 0}\frac{e^{-k}}{k!} \abs[\Big]{ \phi^{(k)}_{F}(t)- \phi^{(k)}_{G(\nu)} (t) }  \le  \sum_{k\ge 0}\frac{e^{-k}}{k!} \rho^{k+1} d_2 (F, G(\nu)) \\
& = \rho \, e^{\rho e^{-1}} d_2 (F, G(\nu)).
\end{align*}
Hence $ \abs{ \phi_{F}(z)  - \phi_{G(\nu)}(z) } \le_{C_{\rho}} d_2 (F, G(\nu))$ for every $z \in \Omega_{\rho,\nu}$. Let $R>0$ such that the disk $D_R \subset \mathbb{C}$ with the origin as center and radius $R$ is contained in the domain $\Omega_{\rho,\nu}$ (note that $R$ depends only on $\nu$, since $\rho$ is a free parameter. For example, one can choose $ \min \{ (2\sqrt{\nu})^{-1}, e^{-1}\} < \rho <2  \min \{ (2\sqrt{\nu})^{-1}, e^{-1}\}$). Now for any $z \in D_R$, and using the fact that
\[
\frac{1}{\phi^2_{G(\nu)}(z)}= \left(e^{2iz} (1-2iz) \right)^\nu,
\]
one can readily conclude that the function $\phi_{G(\nu)}(z)$ is bounded away from $0$ on the disk $D_R$. Also, for any $r \ge 2$,
\begin{equation}\label{eq:lower-2}
\begin{split}
\abs[\big]{ \kappa_r (F) } & \le 2^{r-1}(r-1)! \sum_{i\ge 1} \abs{ c_{i} }^{r}   \le 2^{r-1}(r-1)! \max_{i} \abs{ c_{i} }^{r-2} \sum_{i\ge 1} \abs{ c_{i} }^{2}\\
& \le  2^{r-2}(r-1)! \sqrt{\nu}^{\,r-2} \, \E(F^2) = 2^{r-2}(r-1)! \sqrt{\nu}^{\,r}.
\end{split}
\end{equation}
Therefore, for any $z \in D_R$,
\begin{align*}
\abs[\Big]{ \frac{1}{\phi_{F}(z)} } \le \exp \Big\{ \sum_{r \ge 2} \frac{ \abs{ \kappa_r(F) } }{r!} \abs{z}^r \Big\} & \le  \exp \Big\{ \sum_{r \ge 2} \frac{2^{r-2}(r-1)! \sqrt{\nu}^{\,r}}{r!} \abs{z}^r \Big\}\\
& \le  \exp \Big\{ \sum_{r \ge 2} \frac{2^{r-2}(r-1)! \sqrt{\nu}^{\,r}}{r!} R^r \Big\}:= C_{R,\nu}< \infty.
\end{align*} 
Hence the function $ \phi_{F}(z)$ is also bounded away from $0$ on the disk $D_R$. Also, relation \eqref{eq:lower-2} implies that the following power series (complex variable) converge to some analytic function as soon as $\abs{z} < R$;
\begin{equation}\label{eq:lower-4}
\sum_{r\ge1}\frac{\kappa_r(F)}{r!}(iz)^r, \quad \sum_{r\ge1}\frac{\kappa_r(G(\nu))}{r!}(iz)^r.
\end{equation}
Thus we come to the conclusion that the functions $\phi_{G(\nu)}(z)$ and $\phi_{F}(z)$ are analytic on the disk $D_R$. Moreover, there exists a constant $c >0$ such that $\abs{ \phi_{G(\nu)}(z) },  \abs{ \phi_{F}(z) } \ge c >0$ for every $z \in D_R$. This implies that on the disk $D_R$ there exist two analytic functions $g$ and $g_\nu$ such that
\[
\phi_{F}(z)=e^{g (z)}, \quad \phi_{G(\nu)}(z)=e^{g_\nu (z)},
\]
i.e. $g (z)= \log (\phi_{F}(z))$ and $g_\nu(z)=\log(\phi_{G(\nu)}(z))$, for $z \in D_R$. In fact, the functions $g$ and $g_\nu$ are given by the power series \eqref{eq:lower-4}. Since the derivative of the analytic branch of the complex logarithm is  $(\log z)' = \frac{1}{z}$ (see \cite[Corollary $2.21$]{conway}), one can infer that for some constant $C$ whose value may differ from line to line and for every $z \in D_R$, we have 
\begin{align*}
\abs[\Big]{ \sum_{r \ge 2} \frac{\kappa_r(F) - \kappa_r (G(\nu))}{r!}(iz)^r } & = \abs[\Big]{  \log (\phi_{F}(z)) - \log(\phi_{G(\nu)}(z)) } \\
& \le_C \abs[\Big]{ \phi_{F}(z) - \phi_{G(\nu)}(z) } \le_C d_{2} (F, G(\nu)).
\end{align*}
Now, using Cauchy's estimate for the coefficients of analytic functions, for any $r \ge 3$, we obtain that
\[
\abs[\Big]{ \kappa_r (F) - \kappa_r (G(\nu)) } \le r! R^r \sup_{\abs{z} \le R} \abs[\Big]{ \log \phi_{F} (z) - \log \phi_{G(\nu)}(z) }.
\]
Therefore,
$
\max \Big\{  \abs[\Big]{ \kappa_3 (F) - \kappa_3 (G(\nu)) }, \abs[\Big]{ \kappa_4 (F) - \kappa_4 (G(\nu)) } \Big\} \le_{C} d_{2} (F, G(\nu))$.

\end{proof}

\subsection{Main Result: Non Asymptotic Optimal Gamma Approximation}\label{subsec:optima-gamma}
Now we are ready to present a non asymptotic optimal Gamma approximation in full generality on the second Wiener chaos in terms of the maximum of the third and fourth cumulants. The following result provides an analogous counterpart to the same phenomenon in the case of normal approximation, see \cite[Theorem 1.2]{n-p-optimal} or Theorem \ref{thm:optimal-normal} item (b). 
\begin{thm}\label{thm:optimal-Gamma-2chaos}
	Let $\nu >0$, and $G(\nu) \sim CenteredGamma(\nu)$. Assume that $F= I_2 (f)$ belongs to the second Wiener chaos such that $\E[F^2]=2\nu$. Then there exist two general constants $0<C_1<C_2$ (possibly depending on the parameter $\nu$) such that 
	\begin{equation}\label{eq:optimal-finite-eigenvalue}
	C_1 \, \mathbf{M} (F) \le d_{2} (F, G(\nu)) \le C_2 \,  \mathbf{M} (F).
	\end{equation}
Recall that 
	\[
	\mathbf{M}(F) :=  \max \Big\{ \abs[\Big]{ \kappa_3 (F) - \kappa_3 (G(\nu)) }, \abs[\Big]{ \kappa_4 (F) - \kappa_4 (G(\nu)) }  \Big\}.
	\]
	\end{thm}
	\begin{proof}
	For the upper bound combine Theorem \ref{thm:MainMalliavinSteinBound} with Proposition \ref{prop:Var-estimate-1} estimate \eqref{eq:optimal-Var1}, Proposition \ref{prop:Var-estimate-2} estimate \eqref{eq:optimal-Var2} as well as Lemma \ref{lem:Var(Gamma_r-2Gamma_r-1)} with $r=1$. The lower bound directly follows from Proposition \ref{prop:lower-bound}. 
		\end{proof}
	
	\begin{rem}\label{rem:generalization-higher-chaos}
	In this remark we shortly comment on a natural thought relating to the generalization of the optimal rate \eqref{eq:optimal-finite-eigenvalue} to higher order Wiener chaoses. In addition a complete lack of any non-artificial example of a sequence of random variables in a fixed Wiener chaos of order $q \ge 3$ converging towards the $G(\nu)$ distribution, our investigations imply that such an extension would come at the cost of very complicated computations involving norms of \textit{contraction operators} to verify estimate \eqref{eq:optimal-Var2} (possibly with a different constant). Furthermore, our method to achieve the optimal lower bound, relying on complex analysis, cannot be used anymore in higher order chaoses, and hence one requires the introduction of new ideas.   
	\end{rem}
	
\subsection{Examples}
We start with the following naive example that illustrates the essential role of our operator theory technique to achieve the optimal rate. It is worth mentioning that all the rates achieved in the forthcoming examples are better (by a square power) than those that can be obtained by the Malliavin-Stein bound \cite[Theorem 1.5]{StMethOnWienChaos}. In the following, when $(a_n)_{n\geq 1}$ and $(b_n)_{n\geq1}$ are two non-negative real number sequences, we write $a_n \approx_C b_n$ if $\lim_{n \to \infty} \frac{a_n}{b_n} = C$, for some constant $C>0$.

\begin{ex}\label{ex:naive-example}
	Let $N_1, N_2 \sim \mathscr{N}(0,1)$ be independent. Consider the sequence 
	\begin{equation*}
	\begin{split}
	 F_n &= c_{1,n} \, (N_1^2 -1) + c_{2,n} \, (N_2^2 -1):=   \sqrt{1+ \frac{1}{n}}  (N_1^2 -1) + \sqrt{1 - \frac{1}{n}}  (N_2^2 -1) \\
	& \stackrel{\mathcal{D}}{\longrightarrow} G(2),  \text{ as } n \to \infty.
	\end{split}
	\end{equation*}
First note that $\E[F^2_n]=4$ for every $n \in \N$. Also, using Proposition \ref{Prop:2choas-properties} item $2$, and relation \eqref{eq:CenteredGammaCumulants}, simple computations yield that $ \kappa_4(F_n) - \kappa_4(G(2))= 48\,  \frac{2}{n^2} \approx_C \frac{1}{n^2}$. Similarly $
	 \kappa_3(F_n) - \kappa_3(G(2)) = 8 \sum_{j=1}^{2} \left( c_{j,n}^3 - 1 \right) \approx_C \frac{1}{n^2}$.  Therefore, our main Theorem \ref{thm:optimal-Gamma-2chaos} implies 
	 \begin{equation}\label{eq:Ex1-optimal-rate}
 d_{2} \big(F_n, G(2) \big)  \approx_C \max \Big\{ \abs[\big]{ \kappa_3(F_n) - \kappa_3(G(2)) } , \abs[\big]{ \kappa_4(F_n) - \kappa_4(G(2)) } \Big\} \approx_C \frac{1}{n^2}.
 \end{equation}
The following important remarks are in order. (a) This example represents a typical scenario, in which, in order to obtain the optimal upper bound, one needs to join together two Gamma quantities $\Gamma_3 (F_n) - 2 \Gamma_{2}(F_n)$  and $\Gamma_2 (F_n) - 2 \Gamma_{1}(F_n)$. In fact, it is not difficult, using Lemma \ref{lem:Var(Gamma_r-2Gamma_r-1)}, to see that 
\begin{equation*}
 \Var \left( \Gamma_r (F_n) - 2 \Gamma_{r-1}(F_n) \right) \approx_C  \Var \left( \Gamma_1 (F_n) - 2 F_n \right) \approx_C \frac{1}{n^2},\, \forall \, r\ge 1.
\end{equation*} 
And now consider Remark \ref{rem:Operator-Theory-Essential}. (b) It is classical that the density function $f_n$ of the random variable $F_n$ admits the following explicit representation in terms of \textit{confluent hypergeometric functions}, 
	\[
f_n(x) = \frac{1}{2 \sqrt{c_{1,n} c_{2,n}}} \, e^{ - \frac{x + c_{1,n} + c_{2,n}}{2 c_{1,n}} } \times \confhyper\Big(\frac{1}{2},1, - \frac{c_{1,n} - c_{2,n}}{2 c_{1,n} c_{2,n}} (c_{1,n}+c_{2,n}+x) \Big) \times \ind{x > - c_{1,n} - c_{2,n}}(x).
\]
Also recall that the density of the target $G(2)$ is given by $f_\nu(x) = \frac{1}{2} e^{-\frac{x}{2} -1 } \ind{x > -2}(x)$. Using rather long and tedious computations, one can show that the optimal estimate \eqref{eq:Ex1-optimal-rate} continues to hold in the stronger distance of \textit{total variation}, namely that
\begin{multline*}
	d_{TV}(F_n, G(2))  = \frac{1}{2} \int_{-\infty}^{\infty} \abs{ f_n(x) - f_\nu(x)  } \, dx \\
	\approx_C \max \Big\{ \abs[\big]{ \kappa_3(F_n) - \kappa_3(G(2)) } , \abs[\big]{ \kappa_4(F_n) - \kappa_4(G(2)) } \Big\} \approx_C \frac{1}{n^2}.
\end{multline*}

\end{ex} 
\begin{ex}(U-statistics)\label{ex:1}
	In this example, we consider a second order U-statistic with degeneracy order $1$ inspired by \cite[section 3.1]{a-a-p-s}. The reader may consult the excellent textbook \cite{Serfing} for a general asymptotic theory of $U$-statistics. Let $\{h_i \}_{i \geq 1}$ be an orthonormal basis of $\HH$ and for $i \geq 1$ set $Z_i := I_1(h_i)$. Consider
	\[ U_n = \frac{2}{n (n-1)} \sum_{1 \leq i < j \leq n} Z_i Z_j = I_2 \bigg( \frac{2}{n (n-1)} \sum_{1 \leq i < j \leq n} h_i \widetilde{\otimes} h_j \bigg). \]
	Then $nU_n \stackrel{\mathcal{D}}{\to} G(1)$ as $n \to \infty$ with parameter $\nu=1$. Furthermore to fix the variance to $2\nu=2$, define
	\begin{align*}
	W_n := \sqrt{\frac{n-1}{n}} n U_n & = I_2 \bigg(  \frac{2}{\sqrt{n (n-1)}} \sum_{1 \leq i < j \leq n} h_i \symtensor h_j \bigg) =: I_2(f_n).
	\end{align*}
	We consider the associated Hilbert-Schmidt operator $A_{f_n} g = f_n \cont{1} g$. Using the fact that $(h_i \tensor h_j) \cont{1} h_k = \sprod{h_i, h_k}_{\HH} \, h_j$ we can explicitly compute the non-zero eigenvalues $c_{1,n}, \ldots, c_{n,n}$ of $A_{f_n}$. They are
	\begin{equation}\label{eq:eigenvalues}
	 c_{1,n} = \sqrt{\frac{n-1}{n}}, \text{ and } c_{2,n} = \ldots = c_{n,n} = \frac{-1}{\sqrt{n (n-1)}}. 
	 \end{equation}
	Therefore, as $n \to \infty$, gathering Proposition \ref{Prop:2choas-properties} item $2$, relation \eqref{eq:eigenvalues} and Theorem \ref{thm:optimal-Gamma-2chaos} we get that
		\[ 
	d_2\big(W_n, G(1) \big) \approx_C \abs[\big]{\kappa_3(W_n) - \kappa_3(G(1))} \approx_C \abs[\big]{\kappa_4(W_n) - \kappa_4(G(1))} \approx_C \frac{1}{n}. \]
\end{ex}

In the next example we consider the important problem of the asymptotic behavior of the \textit{least squares} estimators in the \textit{autoregressive} models in the \textit{nearly non-stationary} regime, where the target distribution $G(\nu)$ shows up. For more  details on this fascinating subject, we refer the reader to \cite{C-W,C-W-general,White,Rao,B-C,L-L-S} and references therein when the noise is a martingale difference, and \cite{B-C-fractional} when the innovation process exhibits \textit{long-range} dependence. We also refer to \cite[Proposition 2]{Gotze-Tikhomirov-1} for a study of optimal rates in a general context of quadratic forms. 
\begin{ex}(Least square estimator in nearly non stationary $AR(1)$ model)	
Let $n \in \N$. Let $\beta_n := 1 - \frac{\beta}{n}$. We consider the first order autoregressive process $X_t (n) = \beta_n X_{t-1}(n) + Z_t$, where $t=1,\ldots,n$, $X_0(n)=0$ for all $n$ and $(Z_i)$ is a white noise, i.e. a sequence of i.i.d. $\mathscr{N}(0,1)$ random variables.  It is classical that the least squares estimator of the unknown parameter $\beta_n$, based on discrete observations $X_1(n),\ldots,X_n(n)$, is given by
\[ \widehat{\beta}_n = \frac{\sum_{t=1}^{n} X_{t-1}(n)  X_t (n)}{ \sum_{t=1}^{n} X^2_{t-1}(n)}. \] 
Define $$W^\beta_n:= \frac{2}{\sqrt{n(n-1)}}\left(   \sum_{t=1}^{n} X^2_{t-1}(n) \right) ( \widehat{\beta}_n - \beta_n)= \frac{2}{\sqrt{n(n-1)}} \sum_{i=1}^{n}\sum_{j=1}^{i-1} \beta^{i-j}_n Z_{i} Z_j.$$
Then \cite[Theorem 1]{C-W} implies that as $n \to \infty$:
\begin{equation*}
W^\beta_n \stackrel{\mathcal{D}}{\longrightarrow} 
 W^{\beta}_\infty:= 2 \int_{0}^{1} \left( 1 + t (e^{2\beta} -1) \right)^{-1} B_t dB_t,
\end{equation*} 
where $B=(B_t)_{t \in [0,1]}$ is a standard Brownian motion. 
In particular when $\beta=0$, we observe that $W_\infty:= W^{\beta=0}_\infty = G(1)$ (equality in law), and hence we obtain that $W_n := W^{\beta=0}_n \stackrel{\mathcal{D}}{\longrightarrow} G(1)$. Now, apply Example \ref{ex:1} to deduce that $d_{2} \big(W_n, G(1) \big) \approx_C \frac{1}{n}$.
	\end{ex}

\begin{ex}(Least square estimator in $AR(2)$ model)	
	In this example, we consider the second order autoregressive $AR(2)$ model: 
	\begin{equation}\label{eq:AR2}
	X_n = \beta_{1} X_{n-1} + \beta_{2} X_{n-2} + Z_n,
	\end{equation}
	where $(Z_k)$ is a white noise, and $X_{0}=X_{-1}=0$. Further, assume that the roots of the associated characteristic polynomial $ 1- \beta_1 z - \beta_2 z^2$ are  $e^{i\theta}$ and $e^{-i \theta}$, and lie on the unit disk. Under this condition it is easy to see that $\beta_1 = 2 \cos \theta$ and $\beta_2 =-1$. The least square estimator $\widehat{\boldsymbol{\beta}}_{n} = (\widehat{\beta}_{1,n}, \widehat{\beta}_{2,n})'$ of the parameter $\boldsymbol{\beta}=(\beta_1,\beta_2)' = (2\cos \theta ,-1)'$ for  $n \ge 2$ is given by
	\[ \widehat{\boldsymbol{\beta}}_{n}= \left(  \sum_{k=0}^{n-1} \mathbf{X_k} \mathbf{X_k}' \right)^{-1} \sum_{k=1}^{n}  \mathbf{X_{k-1}} \mathbf{X_k}, \, \text{ where  } \, \mathbf{X_k}=(X_k,X_{k-1})'.  \]
In	\cite{C-W-general}, the asymptotic behavior of  $n (\widehat{\boldsymbol{\beta}}_{n} - \boldsymbol{\beta})= \mathbf{A}^{-1}_n \boldsymbol{b}_n$ has been derived where 

\[  \mathbf{A}_n = \frac{1}{n^2} \sum_{k=2}^{n} 
\begin{bmatrix} 
X^{2}_{k-1} & X_{k-1}X_{k-2}\\ 
X_{k-1}X_{k-2}  & X^{2}_{k-2}
\end{bmatrix}, \text{ and } 
   \boldsymbol{b}_n =   \begin{bmatrix}   b_{1,n} \\ b_{2,n} \end{bmatrix} := \frac{1}{n} \sum_{k=2}^{n} \begin{bmatrix}   X_{k-1}Z_k \\  X_{k-2}Z_k \end{bmatrix} .
\]

	Following \cite[Corollary 3.3.8]{C-W-general}, as $n \to \infty$, one can deduce that
	\[  W^{\theta}_n := 4 ( \cos \theta \, b_{1,n}- b_{2,n}) = 4 \left(  (\cos \theta )	\frac{1}{n} \sum_{k=1}^{n} X_{k-1} Z_k -  	\frac{1}{n} \sum_{k=1}^{n} X_{k-2} Z_k\right)  \stackrel{\mathcal{D}}{\to} G(2).  \]
	
 Note that the sequence $( W^{\theta}_n : n \ge 1)$ belongs to the second Wiener chaos.  An interesting feature of the previous limit theorem is that although the sequence does depend on the parameter $\theta$ in the model, the target distribution is independent of $\theta$. On the other hand, relation \eqref{eq:AR2} together with the assumption $(\beta_1,\beta_2) = (2\cos \theta ,-1)$ yields that
 \[  X_k= \sum_{j=1}^{k} \frac{\sin (k-j+1)\theta}{\sin \theta} Z_j. \]
 Therefore,
\begin{align*}
W^{\theta}_{n} 
& =  \frac{4}{n} \, \sum_{i=2}^{n} \sum_{j=1}^{i-1}  \frac{ \cos \theta \sin (i-j)\theta - \sin (i-j-1)\theta }{\sin \theta}  Z_i Z_j . 
\end{align*} 
By elementary combinatorics, we have for any function $f: \N \to \R$ that 
$ \sum_{i=2}^{n} \sum_{j=1}^{i-1} f(i-j) = \sum_{k=1}^{n-1} (n-k) f(k)$. Using this, and evaluating the sums of sine functions (which are just geometric sums after writing them in terms of complex exponentials), we get
\begin{align}
\E[(W^{\theta}_{n})^2] 
& = \frac{16}{n^2} \bigg\{ \frac{1}{8 \sin^4 \theta} \Big( \cos \big( 2\theta (n-1) \big) - 2 \cos(\theta) \cos\big( \theta (2n-1) \big) + \cos^2(\theta) \cos(2n \theta) \Big) \notag \\
& \hspace{3em} + \frac{1}{8 \sin^2(\theta)} \Big( n \cos(2 \theta) + 1 - n \Big) + \frac{ n(n - 1) }{4} \bigg \}. \label{eq:VarianceOfW_n^theta}
\end{align}
Note that $\big \vert  \kappa_2(W^{\theta}_{n}) - 4 \big \vert \approx_C 1/n$ as $n \to \infty$. Now we scale $W_n^{\theta}$ so that it has variance equal to $4$ for every $n \in \N$. Set $\sigma_n := \sqrt{\Var( W_n^{\theta})}$, and let $\widetilde{W}_n^{\theta} := \frac{2}{\sigma_n} W_n^{\theta} $. 
Using \eqref{eq:SecondWienerChaosCumulantFormulaEigenvalues}, and after some tedious computations, we get that
	\begin{align*}
	\kappa_3( \widetilde{W}_n^{\theta} ) & = \frac{512}{\sigma_n^3 n^3} \bigg( - \frac{3 n \cos^2(n \theta)}{4 \sin^2(\theta)} + \frac{3 \cos^2(n \theta)}{2 \sin^2(\theta)} + \frac{5 n \cos^4(\theta)}{4 \sin^4(\theta)} + \frac{13 n}{4 \sin^2(\theta)} - \frac{3}{2 \sin^2(\theta)} \\
	& \hspace{4.5em} - \frac{5 n}{4 \sin^4(\theta)} + \frac{n^3}{4} - \frac{3 n^2}{2} + \frac{3 n}{4} \bigg).
	\end{align*}
	Using that $\sigma_n^3 \to 8$ as $n \to \infty$, we see that $\lim_{n \to \infty} \kappa_3( \widetilde{W}_n^{\theta} ) = 16 = 8\nu$ (note that $\nu=2$), and furthermore, 
	\[ \abs{  \kappa_3( \widetilde{W}_n^{\theta} ) - \kappa_3(G(2))  } \approx_C \frac{1}{n}. \]
Similar computations yield that  $\abs{  \kappa_4( \widetilde{W}_n^{\theta} ) - \kappa_4(G(2))  } \approx_C 1/n$. Therefore, Theorem \ref{thm:optimal-Gamma-2chaos} can be applied to deduce that $d_2 ( \widetilde{W}_n^{\theta}, G(2)) \approx_C 1/n$. 

\end{ex}

	\begin{ex}(Quadratic forms \cite{w-v} and \cite[section 3.2]{a-a-p-s}) In this example, we consider a general quadratic form in independent standard normal random variables 
		\[
		F_n:= \sum_{1 \le i,j \le n} c_n (i,j) Z_i Z_j, \quad n \in \N,
		\]
		where $C_n=(c_n(i,j))_{1 \le i,j \le n}$ is an $n\times n$ symmetric matrix, and $(Z_i)$ is a sequence of i.i.d standard normal random variables. Let $\nu > 0$ be an integer number. Now, we make the following assumptions: 
		\begin{itemize}
			\item[(a)] The second moment assumption:
			$\sum_{1 \le i,j \le n} c_n(i,j)^2 =\nu, \quad \forall n \in \N$.
			\item[(b)] There exists a sequence $\{b_n^m(i): n,i =1,2,\ldots, m=1,2,\ldots,\nu\}$ of real numbers such that as $n \to \infty$:
			\[  
			\sum_{1 \le i \le n} b_n^m(i)b_n^k(i) \rightarrow \delta_{km}, \text{ and }
			\exists b>0,\  \forall i,m,n,\ \sqrt{n}\mid b_n^m(i)\mid \leq b<+\infty.\]
			\item[(c)] For every $1 \le m\le \nu$, as $n \to \infty$ it holds that:
			$		\sum_{1 \le i,j \le n} c_{n}(i,j) b_n^m(i)b_n^m(j) \rightarrow 1$.
		\end{itemize} 
		Now a direct application of \cite[Theorem 2]{w-v} implies that 	$W_n := F_n - \E[F_n] \stackrel{\mathcal{D}}{\rightarrow} G(\nu)$. Note that $\E[W^2_n]= 2\nu$ for every $n \in \N$ relying on condition (a). Moreover, one can write $W_n = I_2 (\sum_{1 \le i,j \le n} c_n (i,j) h_i \symtensor h_j)$, where $\{h_i \}_{i \geq 1}$ stands for an orthonormal basis of $\HH$, and for $i \geq 1$,as before, we set $Z_i := I_1(h_i)$. Therefore our main Theorem \ref{thm:optimal-Gamma-2chaos} entails that 
		\begin{equation}\label{eq:general-rate-quadratic-forms}
		d_2 (W_n, G(\nu)) \approx_C \max \Big\{ \abs[\big]{ \kappa_3(W_n) - \kappa_3(G(\nu)) } , \abs[\big]{ \kappa_4(W_n) - \kappa_4(G(\nu)) } \Big\}.
		\end{equation} 
		Depending on the particular choice of the matrix $C_n$ in the original quadratic form $F_n$, we can provide explicit rates (in terms of suitable powers of $n$) in the asymptotic relation \eqref{eq:general-rate-quadratic-forms}. For example, following \cite[remark after Theorem 2]{w-v} and \cite[Corollary 3.2]{a-a-p-s}, assume that $\{ e_m : m =1,\ldots,\nu \}$ is a sequence of distinct orthonormal functions in $L^2[0,1]$ such that $e_m \in C^\alpha([0,1])$ for some $\alpha \in (0,1]$. Here $C^\alpha([0,1])$ denotes the space of all H\"older continuous functions with H\"older exponent $\alpha$. Consider the square integrable kernel $K_\nu$ defined as
		\[ K_\nu (x,y)= \sum_{1 \le m \le \nu} e_m(x) e_m (y), \, \forall \, (x,y) \in (0,1)^2.\]
		Finally, for $n \in \N$ and $1 \le i,j \le n$ we set 
		\[
		d_n (i,j):= \frac{1}{n} K_\nu (\frac{i}{n},\frac{j}{n}), \quad \text{and} \quad c_n (i,j):= \sqrt{ \frac{\nu}{\sum_{1 \le i,j \le n} d^2_n (i,j)} } \times d_n(i,j).
		\]
		Now consider the sequence $W_n= F_n - \E[F_n]$ associated to the symmetric matrix $C_n = (c_n (i,j))$ belonging to the second Wiener chaos. Then, it is straightforward to check that the conditions (a)-(b)-(c) are in order with $b^m_n (i)= \frac{e_m(i/n)}{\sqrt{n}}$. On the other hand, it has been shown \cite[Corollary 3.2]{a-a-p-s} that:
		\begin{equation}\label{eq:cumulants-rate}
		\big \vert \kappa_r (W_n) - \kappa_r (G(\nu)) \big \vert \approx_C n^{- \alpha}, \, \forall \, r \ge 2. 
		\end{equation}
		Putting together the asymptotic estimates \eqref{eq:general-rate-quadratic-forms} and \eqref{eq:cumulants-rate}, we obtain the optimal rate $d_2 (W_n,G(\nu)) \approx_C n^{- \alpha}$. Also, the example presented on page $107$ in \cite{StMethOnWienChaos} can be treated in this framework, and resulting in an improved optimal rate of $1/n$.  
	\end{ex}	
	\section{Appendix}
	
	The following lemma provides an explicit representation of the new Gamma operators used in this paper in terms of contractions. Recall that these are not the same as e.g. in \cite{CumOnTheWienerSpace}, but rather the new ones introduced in \eqref{eq:GammOperatorDefinition}.
	
	\begin{lem} \label{lem:RepresentationOfGamma_j}
		For $q \geq 1$, lets $F=I_q(f)$, for some $f \in \HH^{\odot q}$ be an element of the $q$-th Wiener chaos. Then
		\begin{align}
		\Gamma_{s}(F)  = & \sum_{r_1=1}^{q} \cdots \sum_{r_{s}=1}^{[sq - 2 r_1 - \cdots - 2 r_{s-1}] \wedge q} c_q (r_1, \ldots, r_{s}) \ind{r_1<q} \ldots \ind{r_1 + \cdots + r_{s-1} < \frac{sq}{2}} \notag \\
		& \times I_{(s+1)q - 2 r_1 - \cdots - 2 r_{s} } \left( \left( \left( \ldots (f \scont{r_1} f ) \scont{r_2} f  \right) \ldots f  \right) \scont{r_{s}} f \right), \label{eq:NewGammaOpProof1}
		\end{align}
		where the constants $c_q(r_1, \cdots, r_{s})$ are recursively defined via
		$ c_q(r) = q \, (r-1)! \, \binom{q-1}{r-1}^2$, and for $s\ge 2$, 
		\begin{align} 
		& c_q(r_1, \cdots, r_{s}) =  \notag \\
		& (sq- 2 r_1 - \cdots - 2 r_{s-1}) \, (r_s-1)!  \, \binom{sq-2r_1- \cdots -  2r_{s-1} -1 }{r_s - 1} \binom{q-1}{r_s - 1} \,  c_q(r_1,  \cdots, r_{s-1}). \label{eq:RecursiveFormulaForc_q}
		\end{align}
	\end{lem}
	\begin{proof} It follows by induction on $s$ and similar lines of arguments as in \cite[Proof of Theorem 5.1]{CumOnTheWienerSpace}. 
		
	\end{proof}
	
	
	\begin{proof}[Proof of Proposition \ref{Prop:RelationOldAndNewGamma}]
		Part \textit{(a)} is clear from the definition. Part \textit{(b)} for $j = 1$ is also trivial. For $j = 2$, we use the fact that $\Gamma_1 = \Gamma_{alt,1}$, as well as the integration by parts formula \eqref{eq:IntegrationByParts}, to get
		\begin{align*}
		\E \big[ \Gamma_2(F) \big] & = \E \big[ \sprod{D \Gamma_1(F), -D L^{-1} F}_{\HH} \big] = \E\big[ \Gamma_1(F) \, F \big] \\
		& = \E \big[F \, \Gamma_{alt,1}(F) \big] =  \E \big[ \sprod{D F, -D L^{-1} \Gamma_{alt,1}(F)}_{\HH} \big] = \E[\Gamma_{alt,2}(F)].
		\end{align*}
		For part \textit{(c)}, consider
		\begin{align*}
		\E \big[ \Gamma_3(F) \big] & = \E \big[ \sprod{D \Gamma_2(F), -D L^{-1} F}_{\HH} \big] = \E \big[ F \, \Gamma_2(F) \big] = \E \big[ F \, \sprod{D \Gamma_1(F), -D L^{-1} F}_{\HH} \big] \\
		& = \E \big[ \sprod{D \big(F \,\Gamma_1(F) \big), -D L^{-1} F}_{\HH} \big] - \E \big[ \Gamma_1(F) \sprod{D F, -D L^{-1} F}_{\HH} \big] \\
		& = \E \big[ \sprod{D \big(F \,\Gamma_1(F) \big), -D L^{-1} F}_{\HH} \big] - \E \big[ \Gamma_{alt,1}(F)^2 \big] \\
		& = \E \big[ F^2 \, \Gamma_{alt,1} \big] - \E \big[ \Gamma_{alt,1}(F)^2 \big] \\
		& = \E[F^2] \, \E\big[ \Gamma_{alt,1}(F) \big] + \E \big[ 2F \, \sprod{D F, -D L^{-1}  \Gamma_{alt,1}(F)}_{\HH}  \big] - \E \big[ \Gamma_{alt,1}(F)^2 \big] \\
		& = \E\big[ \Gamma_{alt,1}(F) \big]^2 + 2 \, \E \big[ F \, \Gamma_{2,alt} \big] - \E \big[ \Gamma_{alt,1}(F)^2 \big] \\
		& = - \Var\big(\Gamma_{alt,1}(F)\big) + 2 \, \E\big[ \Gamma_{alt,3}(F) \big].
		\end{align*}
		For part \textit{(d)}, we consider the representation of $\Gamma_{alt,s}$ given in equation (5.25) of \cite{CumOnTheWienerSpace}. The representation is exactly the same as for $\Gamma_s$ (Lemma \ref{lem:RepresentationOfGamma_j}), except for the recursive formula of the constants $c_q$. For $\Gamma_{alt,j}$ they are given by $c_{alt,q}(r) = c_q(r) = q (r-1)! \binom{q-1}{r-1}^2$, and for $s\ge2$,
		\[ c_{alt,q}(r_1, \ldots, r_s) = q \, (r_s-1)!  \, \binom{sq-2r_1- \cdots -  2r_{s-1} -1 }{r_s - 1} \binom{q-1}{r_s - 1} \,  c_q(r_1,  \cdots, r_{s-1}). \]
		Comparing this with our formula \eqref{eq:RecursiveFormulaForc_q}, we see that only the first factor is different, namely $q$ instead of $(sq-2r_1-\ldots - 2r_{s-1})$. But now for $q=2$, the indicator $\ind{r_1 + \cdots + r_{s-1} < \frac{sq}{2}}$ dictates that $r_1 = \ldots = r_{s-1} = 1$. Hence $q = 2 = 2s-2r_1 - \ldots - 2 r_{s-1}$. Therefore, the two notions of Gamma operators coincide when $q=2$.
	\end{proof}

\section*{Acknowledgments}
The authors would like to thank Simon Campese for pointing out a mistake in the proof of Theorem \ref{thm:MainMalliavinSteinBound}.

\end{document}